\documentclass{amsart}

\usepackage{epsfig}
\usepackage{amsmath}
\usepackage{amssymb}
\usepackage{amsthm}
\usepackage{verbatim}
\usepackage{boxedminipage}
\usepackage[all, knot]{xy}
\xyoption{arc} 

\newtheorem{theorem}{Theorem}[subsection]
\newtheorem{proposition}[theorem]{Proposition}
\newtheorem{corollary}[theorem]{Corollary}
\newtheorem{lemma}[theorem]{Lemma}

\theoremstyle{definition}
\newtheorem{definition}[theorem]{Definition}

\newcommand{\To}{\longrightarrow}
\newcommand{\into}{\hookrightarrow}

\newcommand{\isoto}{\stackrel{\sim}{\To}}

\newcommand{\G}{\mathbb{G}}
\newcommand{\rbar}{\bar{r}}
\newcommand{\C}{\mathbb{C}}
\newcommand{\A}{\mathbb{A}}
\newcommand{\Pj}{\mathbb{P}}
\newcommand{\Z}{\mathbb{Z}}
\newcommand{\bbX}{\mathbb{X}}
\newcommand{\Q}{\mathbb{Q}}
\newcommand{\Qbar}{\overline{\mathbb{Q}}}

\newcommand{\cH}{\mathcal{H}}
\newcommand{\Hcan}{\mathcal{H}^{\mathrm{can}}}
\newcommand{\lambdacan}{\lambda^{\mathrm{can}}}
\newcommand{\cL}{\mathcal{L}}
\newcommand{\cA}{\mathcal{A}}
\newcommand{\cB}{\mathcal{B}}
\newcommand{\cG}{\mathcal{G}}
\newcommand{\cF}{\mathcal{F}}
\newcommand{\cP}{\mathcal{P}}
\newcommand{\cW}{\mathcal{W}}
\newcommand{\fF}{\mathfrak{F}}
\newcommand{\fl}{\mathfrak{l}}
\newcommand{\fq}{\mathfrak{q}}
\newcommand{\fQ}{\mathfrak{Q}}
\newcommand{\fw}{\mathfrak{w}}
\newcommand{\Gal}{\mathrm{Gal}\,}

\newcommand{\Frob}{\mathrm{Frob}}
\newcommand{\BDR}{B_{\mathrm{DR}}}
\newcommand{\Fil}{\mathrm{Fil}}
\newcommand{\gr}{\mathrm{gr}}
\newcommand{\Hom}{\mathrm{Hom}\,}
\newcommand{\sgn}{\mathrm{sgn}\,}
\newcommand{\Spec}{\mathrm{Spec}\,}
\newcommand{\Br}{\mathrm{Br}\,}
\newcommand{\Frac}{\mathrm{Frac}\,}
\newcommand{\Prim}{\mathit{Prim}}
\newcommand{\List}{\mathit{List}}

\newcommand{\Tr}{\mathrm{Tr}}
\newcommand{\ad}{\mathrm{ad}}
\newcommand{\GSp}{\mathrm{GSp}}
\newcommand{\Sp}{\mathrm{Sp}}
\newcommand{\GL}{\mathrm{GL}}
\newcommand{\PGL}{\mathrm{PGL}}
\newcommand{\SL}{\mathrm{SL}}
\newcommand{\ssm}{{^{\mathrm{ss}}}}
\newcommand{\Fbar}{\overline{F}}
\newcommand{\Kbar}{\overline{K}}
\newcommand{\F}{\mathbb{F}}
\newcommand{\eps}{\epsilon}
\newcommand{\bfk}{\mathbf{k}}
\newcommand{\notdiv}{\!\!\not|\,}

\newcommand{\oldbegstrat}{
\par\vspace{-0.3\baselineskip}\hfill\begin{boxedminipage}[t]{4.7in}
\hspace{-0.3in}\begin{minipage}{4.9in}
\begin{description}
}

\newcommand{\ra}{\rightarrow}

\newcommand{\oldstratdone}{
\end{description}
\end{minipage}
\end{boxedminipage}\hfill\vspace{0.3\baselineskip}\par
}

\newcommand{\begstrat}{
\begin{figure}[ht]
\hfill\begin{boxedminipage}[t]{4.7in}
\hspace{-0.15in}\begin{minipage}{4.7in}
\begin{itemize}
}
\newcommand{\stratdone}{
\end{itemize}
\end{minipage}
\end{boxedminipage}\hfill
}
\newcommand{\stratalldone}{
\end{figure}
}

\title{Potential automorphy for certain Galois representations to $\GL_{2n}$}
\author{Thomas Barnet-Lamb}
\email{tbl@math.harvard.edu} 
\address{Department of Mathematics\\Harvard University\\Cambridge\\MA 02138\\USA} 

\begin{document}
\subjclass{11R39 (primary), 11F23 (secondary)}
\keywords{Galois representation, potential automorphy, potential modularity, Dwork hypersurface}
\thanks{The author was partially supported by NSF grant DMS-0600716 and by a Jean E.~de Valpine Fellowship.} 

\begin{abstract}
Building upon work of Clozel, Harris, Shepherd-Barron, and Taylor, this paper shows that certain Galois representations become automorphic after one makes a suitably large totally-real extension of the base field. The main innovation here is that the result applies to Galois representations to $\GL_{2n}$, where previous work dealt with representations to $\GSp_n$. The main technique is the consideration of the cohomology the Dwork hypersurface, and in particular, of pieces of this cohomology other than the invariants under the natural group action.     
\end{abstract}
\maketitle

\section{Introduction}

\subsection{} The aim of this document is to prove a potential automorphy theorem: that is, a statement that certain Galois representations become automorphic when we make a large field extension. I will begin by first introducing just enough definitions to state the theorem which I will prove, and proceed to state it.

The first notion we will need to define is the notion of the \emph{sign} of a polarizable Galois representation, after Bella\"iche-Chenevier (see \cite[\S1.1]{bc}). For $l$ a rational prime, we will write $\eps_l$ to denote the $l$-adic cyclotomic character.
\begin{definition} \label{def:bc-sign}
 Let $F$ be a CM\footnote{For us, `CM field' will always mean imaginary CM field.} or totally real field, $l$ a rational prime, $r:\Gal (\Fbar/F) \ra \GL_n(\Z_l)$ a representation, and $\chi:\Gal (\Fbar/F) \ra \Z_l^\times$ a character. We say that $r$ is \emph{essentially conjugate self dual with similitude factor $\chi$}, if there exists an isomorphism $r^c \cong r^\vee\otimes\chi$. (We will be most interested in the case where $\chi=\eps_l^{1-n}$, and in this case we simply call $r$ \emph{conjugate self dual}.) By a \emph{polarized} representation (with similitude factor $\chi$) we mean a representation which is essentially conjugate self dual with similitude factor $\chi$, equipped with a specific choice of isomorphism $r^c \cong r^\vee\otimes\chi$. (We sometimes call the choice of isomorphism the \emph{polarization}.)
 
 We can think such an isomorphism as giving us a pairing $\langle*,*\rangle$ on $(\Z_l)^n$ satisfying $\langle r(\sigma) v_1,r({}^c\sigma) v_2\rangle=\chi(\sigma)\langle v_1,v_2\rangle$ for each $\sigma\in \Gal (\Fbar/F)$ and $v_1,v_2\in(\Z_l)^n$. If $r$ is in addition assumed to be absolutely irreducible, this pairing will either be symmetric or antisymmetric. We define the \emph{sign} of $r$ to be +1 if the pairing is symmetric, -1 if it is antisymmetric, and write $\sgn r$ for the sign of $r$. 
\end{definition}

The point of this definition is that just as two dimensional Galois representations come in two kinds, odd and even, with radically different properties (odd representations are generally well-behaved and even representations are a mystery), there is a similar dichotomy for higher-dimensional representations. This is what is captured by the Bella\"iche-Chenevier sign. Those representations with sign +1 are the `good' ones (generalizing odd two dimensional representations), and it will come as little surprise that we will have to restrict our theorems to such representations. (The theorems of \cite{hsbt} contain such a restriction implicitly, since they deal only with symplectic Galois representations with totally odd multiplier, which will have sign +1 automatically.)  

More precisely, suppose $F$ is a totally real field, $K$ a totally imaginary quadratic extension, and $\Pi$ is a regular algebraic, conjugate self dual, cuspidal automorphic representation of $\GL_n$ over $K$. Then there is a Galois representation associated to $\Pi$ by the work of Shin \cite{shin} and the many coauthors of the Paris book project \cite{parisbook}. Bella\"iche and Chenevier prove (see \cite[Theorem 1.2]{bc}) that the Galois representation associated to $\Pi$ will have sign +1 in the sense they define. Thus, since our aim when we prove our potential automorphy theorem is to start with a Galois representation $r$ and eventually find a $\Pi$ whose corresponding Galois representation (in the sense just discussed) is a restriction of $r$, our objective can only be possible if the restriction of $r$ (and hence $r$ itself) have sign +1.

For further information on the Bella\"iche-Chenevier sign, see the next subsection, where we make explicit the connection between this sign and the usual notion of `oddness' for a two-dimensional representation.

\smallskip Since we restrict our attention to polarized representations (in the sense of Definition \ref{def:bc-sign}) with similitude factor $\eps_l^{1-n}$, with sign +1, and \emph{with $\Z_l$ coefficients}, there is some further information that we can extract. The polarization gives us a symmetric pairing on the underlying vector space $V$ of the representation, and we can reduce mod $l$ to get a symmetric pairing on the $\F_l$ vector space $V\otimes\F_l$. Such a pairing has an associated invariant called the \emph{determinant}, which is a well defined element of $\F_l^\times/(\F_l^\times)^2$, the multiplicative group of elements of $\F_l$ modulo squares. 

Note that it is important to distinguish the determinant of $r$, which is a character of $G_F$, from the determinant of the pairing associated to the polarization of $\bar{r}$, which is an element of $\F_l^\times/(\F_l^\times)^2$. 
\begin{definition} Given a polarization on a representation $r$ as above, the \emph{determinant of the polarization} will refer to the determinant of the pairing associated to the polarization of $\bar{r}$. We will say the polarization \emph{has square determinant} if this determinant is the identity element of $\F_l^\times/(\F_l^\times)^2$. 
\end{definition}

This determinant of the polarization will add a technical restriction to our theorem: we will only be able to prove a representation $r$ potentially modular when the determinant of the polarization of $r$ is a square.\footnote{In case the reader is confused as to why one doesn't simply side-step this restriction by extending scalars (for the coefficients) to \emph{ensure} that the polarization determinant is a square, we remark that this is not allowed by one of the other conditions of our theorem. Specifically, our theorem, like that of \cite{hsbt}, requires that the coefficients of the representation are $\Q_l$ and not any extension field.} It is worth remarking that while the sign +1 restriction reflects a deep reality in Galois representations, the restriction on the determinant of the polarization appears to be a relatively shallow technical problem: for instance, the polarization determinant invariant becomes meaningless if we allow extension of the field of coefficients. Thus one might hope that this restriction might be removed in future work.

Finally, we recall that given a number field $F$, a finite set $S$ of places of $F$, a rational prime $l$, and a Galois representation $r: G_F\to \GL_n(\Q_l)$, we say that $\rho$ is \emph{automorphic of type $\{\Sp_n(1)\}_{v\in S}$} if there is an RAESDC representation $\Pi$ of $\GL_n(\A_F)$ of weight 0, whose local component at every place in $S$ is an unramified twist of the Steinberg representation, such that for all finite places $v$ of $F$, the Weil-Deligne representation associated to the restriction of $r$ to the decomposition group at $v$ is associated to the local component of $\Pi$ at $v$ via the local Langlands correspondence. 

We are now in a position to state our main theorem. 

\begin{theorem} \label{orig-main-theorem}
For each pair of positive integers $n,N$ with $N\geq n+5$, $n$ even, and $N$ odd, we can find a constant $C(n,N)$ and a quadratic extension $F^*(n,N)$ of $\Q(\mu_N)$ with the following property\footnote{The astute reader will note that since $N\geq n+5$, we could consider the constant $C$ as just depending on $N$, by taking an appropriate maximum over $n$. Nonetheless, I have chosen to emphasize $n$, which is in some sense much more important than $N$, by leaving it in the notation.}:

Suppose that $F$ is a CM field containing $\mu_N$. Suppose that $l>C(n,N)$ is a rational prime which is unramified in $F$ and $l\equiv 1 \mod N$. Suppose in addition that $l$ splits in $F^*(n,N)$. Let $v_q$ be a prime of $F$ above a rational prime $q\neq l$ such that  $q\notdiv N$. Let $\cL$ be a finite set of primes of $F$ not containing primes above $lq$.

Suppose that we are given a representation 
$$r:\Gal (\Fbar/F) \ra \GL_n(\Z_l)$$
enjoying the following properties:
\begin{enumerate}
\item $r$ ramifies only at finitely many primes.
\item $r^c \cong r^\vee\eps_l^{1-n}$
\item $r$ has sign +1, the the sense of Bella\"iche-Chenevier.
\item For each prime $\fw|l$ of $F$, $r|_{\Gal(\Fbar_\fw/F_\fw)}$ is crystalline with Hodge-Tate numbers $\{0,1,\dots, n-1\}$
\item $r$ is unramified at all the primes of $\cL$
\item $(r|_{\Gal(\Fbar_{v_q}/F_{v_q})})\ssm $ and $\bar{r}|_{\Gal(\Fbar_{v_q}/F_{v_q})}$ are unramified, with $(r|_{\Gal(\Fbar_{v_q}/F_{v_q})}) \ssm$ having Frobenius eigenvalues $1, (\#k(v_q)), \dots, (\#k(v_q))^{n-1}$
\item $\det \bar{r}\cong \eps_l^{n(1-n)/2}$ mod $l$
\item $\bar{r}|_{\Gal(\Fbar/F(\zeta_l))}$ is `big'.\footnote{Or more precisely, if we let $r'$ denote the extension of $r$ to a continuous homomorphism $\Gal(\Fbar/F^+)\ra\cG_n(\Qbar_l)$ as described in section 1 of \cite{cht}; then $\bar{r}'|_{\Gal(\Fbar/F(\zeta_l)}$ is `big'.}
\item $\Fbar^{\ker\ad \bar{r}}$ does not contain $F(\zeta_l)$
\item $\bar{r}$ satisfies, for each prime $\fw|l$ of $F$:
$$\bar{r}|_{I_{F_\fw}} \cong 1 \oplus \eps_l^{-1} \oplus \dots \oplus  \eps_l^{1-n} $$
\item $\bar{r}$ admits a polarization with determinant a square.
\end{enumerate}

Then there is a CM field $F'$ containing $F$ and linearly independent from  $\Fbar^{\ker \bar{r}}$ over $F$. In addition, all primes of $\cL$ and all primes of $F$ above $l$ are unramified in $F'$. Finally, there is a prime $w_q$ of $F'$ over $v_q$ such that $r|_{\Gal(\Fbar/F')}$ is automorphic of weight 0 and type $\{\Sp_n(1)\}_{\{w_q\}}$.

Moreover, if at the same time we are given $F$ we are given a CM subfield $F_0$ of $F$ which also contains $\mu_N$, then we can additionally arrange that $F'$ is Galois over $F_0$.
\end{theorem}

This theorem generalizes work of Harris, Shepherd-Barron and Taylor. The key advances in this work are
\begin{itemize}
\item The representation $r$ can now map into $\GL_n$; in the earlier work, it was required to map into $\GSp_n$. (Restrictions were also placed on the multiplier.)
\item The ability to vary the integer $N$ is new. In the earlier work, $n+1$ replaces $N$ in all conditions above which refer to $N$, and no integer $N$ is mentioned. This makes these conditions significantly more restrictive, for instance, the older theorem requires that $l\equiv 1 \mod n+1$.
\end{itemize}

This paper relies heavily on work of Katz in \cite{k} and on the lifting theorems of Clozel, Harris and Taylor in \cite{cht}. The question of looking at other parts of the cohomology of the Dwork hypersurface was raised by Guralnick, Harris and Katz in \cite{ghk}.

\subsection{The Bella\"iche-Chenevier sign}\label{sec:bc sign}
As was explained above, one of the key conditions in the main theorem of this paper concerns the Bellai\"che-Chenevier sign of the representation which we would like to prove modular; the theorem only applies to representations with sign +1. It was also mentioned that the condition that the sign be +1 extends the familiar notion of oddness for a 2-dimensional Galois representation over a totally real field. (If $F$ is a totally real field and $r:G_F\to\GL_2(\Q_l)$, we call $r$ \emph{odd} if $\det(r(c))=-1$ for any complex conjugation $c$ in $G_F$.) Since we anticipate that the notion of the  Bellai\"che-Chenevier sign may be somewhat unfamiliar, we will pause now to explain the connection to the familiar notion of oddness in detail. 

Before we can describe this connection, however, we must introduce some further notions, which the reader should compare to those in Definition \ref{def:bc-sign}.
\begin{definition} 
Let $F$ be a number field (usually, for our purposes, totally real), $l$ a rational prime, $r:\Gal (\Fbar/F) \ra \GL_n(\Z_l)$ a representation, and $\chi:\Gal (\Fbar/F) \ra \Z_l^\times$ a character. We say that $r$ is \emph{essentially self dual (with similitude factor $\chi$)}, if there exists an isomorphism $r \cong r^\vee \otimes\chi$. 
\end{definition}
If $r$ is an essentially self dual representation, we can think of the choice of an isomorphism as in the definition as being the same as giving a pairing $\langle|*,*|\rangle$ on $(\Z_l)^n$ satisfying $\langle| r(\sigma) v_1,r(\sigma) v_2|\rangle=\chi(\sigma)\langle| v_1,v_2|\rangle$ for each $\sigma\in \Gal (\Fbar/F)$ and $v_1,v_2\in(\Z_l)^n$; we use the slightly cumbersome notation $\langle|*,*|\rangle$ in order to visually distinguish pairings arising from essential self duality from those arising form essential conjugate self duality, since (as we will soon see) it is possible for a single representation space to have both kinds of pairing simultaneously. As with conjugate self duality, if $r$ is absolutely irreducible, the $\langle|*,*|\rangle$ pairing will either be symmetric or antisymmetric. We define the \emph{SD-sign} of $r$ to be +1 if the pairing is symmetric, -1 if it is antisymmetric, and write $\sgn_{SD} r$ for this SD-sign. (Again, the cumbersome notation is to make clear the distinction between this notion of sign and the notion introduced in Definition \ref{def:bc-sign} for (essentially) conjugate self dual representations.)

Usually, it is the case that `most Galois representations are not self-dual', but for two dimensional representations, \emph{all} Galois representations are self dual. Concretely, we can put a natural symplectic pairing on $(\Z_l)^2$ given by $\langle|v_1,v_2|\rangle\mapsto \det(v_1|v_2)$, where $(v_1|v_2)$ is a matrix with columns $v_1$ and $v_2$; then this pairing will have the property that $\langle|Mv_1,Mv_2|\rangle=\det M\langle|v_1,v_2|\rangle$ for any endomorphism $M$. Thus if $r$ is a two dimensional Galois representation, we will have $\langle|r(\sigma)v_1,r(\sigma)v_2|\rangle=\det r(\sigma)\langle|v_1,v_2|\rangle\mapsto \det(v_1|v_2)$, and hence $\langle|*,*|\rangle$ gives an isomorphism $r \cong r^\vee \otimes(\det r)$ exhibiting $r$ as essentially self dual with similitude factor $\det r$. Notice that since this pairing is symplectic, we always have $\sgn_{SD} r=-1$.

Now, let us consider Galois representations of arbitrary dimension, but study the case where the representation is defined over a totally real field: let us say $F$ is a totally real field and $r:\Gal (\Fbar/F) \ra \GL_n(\Z_l)$. In this case, the group automorphism of $\Gal(\Fbar/F)$ given by conjugation by a complex conjugation $c$ will be an inner automorphism and hence we see $r^c\cong r$. It follows that the notion of an essentially self dual with similitude factor $\chi$ (satisfying $r \cong r^\vee \otimes\chi$) coincides in this case with the notion of an essentially conjugate self dual representation (satisfying $r^c \cong r^\vee \otimes\chi$). Concretely, given a pairing  $\langle|*,*|\rangle$ encoding the essential self duality (and therefore satisfying satisfying $\langle| r(\sigma) v_1,r(\sigma) v_2|\rangle=\chi(\sigma)\langle| v_1,v_2|\rangle$), we can define a pairing $\langle*,\rangle$ by 
$$\langle v_1,v_2\rangle=\langle|  v_1,r(c) v_2|\rangle$$
where $c$ is any choice of complex conjugation. We see then that this pairing will satisfy $\langle r(\sigma) v_1,r({}^c\sigma) v_2\rangle=\chi(\sigma)\langle v_1,v_2\rangle$ and so encodes an essential conjugate self duality. (There is of course a similar formula for passing from the essential conjugate self duality pairing to the essential self duality pairing.)

The reason we introduce these somewhat explicit formulae is that we can now relate the sign of $r$ with its SD sign, as follows. We have that  $\langle v,w\rangle=\sgn(r)\langle w,v\rangle$; but on the other hand
\begin{align*}
\langle v,w\rangle &= \langle|  v,r(c) w|\rangle=\sgn_{SD} (r)\langle|  r(c)w, v|\rangle
 				=\sgn_{SD} (r)\langle| r(c)w, r(c)^2 v|\rangle \\
 				&=\sgn_{SD}(r) \chi(c) \langle| w, r(c) v|\rangle 
				=\sgn_{SD}(r) \chi(c) \langle w, v\rangle 
\end{align*}
It follows that $\sgn(r)=\sgn_{SD}(r) \chi(c)$: that is, the sign of $r$ is the product of the SD sign and the value of the multiplier on complex conjugations.

Thus we have seen that for two dimensional representations (of the Galois group of any number field) there is a natural essential self-duality pairing with SD-sign -1. We have also seen that for representations of totally real fields (of any dimension), one can interconvert essential self duality pairings and essential conjugate self-duality pairings, and related their signs. In the case of a two dimensional representation over a totally real field, we apply both of these ideas. We see that a two dimensional representation of the Galois group of a totally real field will automatically have an essential self-duality pairing with similitude factor $\chi=\det r$ and SD-sign -1; this will in turn give us a conjugate essential self-duality pairing with sign $-1\times\chi(c)=-(\det r)(c)$. 

From this it follows that the Bella\"iche-Chenevier sign will be +1 if and only if $(\det r)(c)=-1$, which is the classical definition of an odd 2 dimensional representation.

\subsection{The strategy}\label{sec:strat}

I will now describe the strategy of the argument. There are several stages:
\begin{enumerate}
\item We begin by introducing the Dwork family. For each integer $N$, the Dwork family $Y\subset \Pj^{N-1}\times \Pj^1$ is a projective family of hypersurfaces over $\Pj^1$, smooth over $\A^1\backslash\mu_N$, where $\mu_N$ denote the $N$th roots of 1. It has general equation $\nu(X_1^{N}+X_2^{N}+\dots+X_N^{N}) = N\lambda X_1 X_2\dots X_N$ where $(\nu:\lambda)\in\Pj^1$ is the parameter. Much of our work will consist of studying the relative cohomology of this family over the base $\Pj^1$.
\item We introduce a action of the group $(\mu_N)^N$ on this family, and use this action (and results of Katz) to decompose the relative cohomology of the family with $l$-adic coefficients into pieces. Having chosen an even integer $n<N-4$, we single out one of these pieces, $\Prim^{N-2}_{l,[v]}$. This piece is an \'etale sheaf on $\Pj^1$, of dimension $n$, and lisse over $\A^1\backslash\mu_N$. We study various properties of this piece. Most importantly, we \emph{(a)} calculate the Hodge-Tate numbers (showing that they form an unbroken sequence without gaps or repetitions), and \emph{(b)} study the monodromy of the sheaf, showing that it is the entire special linear group. (We establish various other properties of lesser importance.) Finally, we deduce similar facts about the monodromy of the corresponding piece of the cohomology with mod $M$ coefficients, showing that this is also the full special linear group as long as $M$ is divisible only primes above a certain bound $C(n,N)$.
\item We then consider the following question: given two rational primes $l$ and $l'$, a mod $l$ Galois representation $\bar{r}$, and another mod $l'$ Galois representation $\bar{r}'$, both over a CM field $K$, can we find a variety in our family `whose mod $l$ cohomology looks like $\bar{r}$ and whose mod $l'$ cohomology looks like $\bar{r}'$'? More formally, we ask if we can find two things: first, a CM extension $K'$ of $K$, linearly disjoint from the fixed field of the kernel of $\bar{r}$ (and similarly disjoint from the fixed field of the kernel of $\bar{r}'$), and unramified at $l$ and $l'$. And second, a point $t\in (\A^1\backslash\mu_N)(K')$ such that\footnote{Actually, further conditions are imposed on this field, but we focus on the most important ones in this sketch.} the fiber of the Dwork family over $t$ has mod $l$ cohomology which agrees with $\bar{r}|_{G_{K'}}$ (as a Galois representation of $G_{K'}$) and mod $l'$ cohomology which agrees with $\bar{r}'|_{G_{K'}}$. We show that the answer is `yes' under certain conditions, using the following strategy. First, we reduce the problem to showing that a certain variety defined over the totally real subfield of $K$ has a point over some totally real extension. Then, we use the theorem of Moret-Bailly (a general theorem of arithmetic geometry which lets one find points over extension fields) to show that this follows from the geometric irreducibility of the variety in question, and a series of local checks at various places. The geometric irreducibility can then be deduced from the results on the monodromy of $\Prim^{N-2}_{l,[v]}$ established in the previous part, while the sundry local checks turn out to be straightforward given the other properties we established.
\item Then, for each prime $l'$ and CM field $K$ unramified at $l'$, we establish the existence of mod $l'$ Galois representations $\bar{r}'$ with the following property. Any $l'$-adic Galois representation:
\begin{itemize}
\item defined over a CM field extending $K$, unramified at $l'$ and linearly disjoint from $K^{\ker \bar{r}'}$,
\item which agrees with $\bar{r}'$ mod $l'$,
\item whose Hodge-Tate numbers form an unbroken sequence without gaps or repetitions,
\item and which satisfies certain other properties of lesser importance,
\end{itemize}
will automatically be automorphic. The basic method by which this is done is to establish a good supply of representations which are automatically known to be modular over any extension field---the ultimate source of these representations being the automorphic induction of characters---and then to apply automorphy lifting theorems; but there are complications involving the need of so-called `Steinberg hypotheses'. Luckily we need not pay too great attention to the technical issues here, since the construction of the representations that we need has already been carried out in \cite{hsbt,cht}, and we may simply cite the appropriate portions of their work.
\item We can now put the results we have established together to prove our main theorem. Given an integer $n$, a rational prime $l>C(n,N)$, a CM field $K$, and an $l$-adic Galois representation $r:G_K\to\GL_n(\Z_l)$ of dimension $n$ which we would like to prove potentially modular, we find (using the previous item) an $\bar{r}'$ with the property described there. We then use point (2) to find a CM extension $K'$ of $K$, unramified at $l$ and $l'$ and linearly disjoint from the fixed fields of the kernels of $\bar{r}$ and $\bar{r}'$, and a point $t\in (\A^1\backslash\mu_N)(K')$, such that the fiber of the Dwork family over $t$ ($Y_t$, say) has mod $l$ cohomology which agrees with $\bar{r}|_{G_{K'}}$ (as a Galois representation of $G_{K'}$) and mod $l'$ cohomology which agrees with $\bar{r}'|_{G_{K'}}$. By the key property of $\bar{r}'$ discussed in the previous point, this allows us to deduce that the $l'$-adic cohomology of $Y_t$ is automorphic. (Note that here we use the fact about the Hodge-Tate numbers of $\Prim^{N-2}_{l',[v]}$ established in point (1).) Tautologically we deduce that the $l$-adic cohomology is autmorphic too. Finally, we apply a modularity lifting theorem (again using the fact about the Hodge-Tate numbers from point (1)) to deduce that $r|_{G_{K'}}$ is automorphic, since it agrees mod $l$ with $\Prim^{N-2}_{l,[v]}$ .
\end{enumerate}
Of course, the sketch above omits many minor details (for instance, the theorems we apply generally have many conditions, and we must carefully manage the bookkeeping to ensure that we always have the conditions we require when we wish to apply such a theorem). Nonetheless, it gives the main shape of the argument.

I will now describe the organization of the remainder of the paper. In section \ref{sec:geom} we study the geometry of the Dwork family, discussing the properties which we will require in our argument (steps (1) and (2) in the strategy above). In section \ref{sec:moret-bailly} we carry out the argument, using the Theorem of Moret-Bailly, which allows us to find varieties in the Dwork family over large extension fields whose cohomology mod $l$ and mod $l'$ agrees with the restrictions of mod $l$ and mod $l'$ representations $\bar{r}$ and $\bar{r}'$  we have been given independently (step (3) in the strategy above). In section \ref{sec:seed} we find mod $l'$ representations that we can use to deduce $l'$-adic representations modular (step (4)), and finally in section \ref{sec:final} we put the pieces together to prove our main theorem.

{\bf Author's note:} An earlier version of this paper erroneously claimed a version of Theorem \ref{orig-main-theorem} with condition (7) replaced by the weaker condition $(\det \bar{r})^2\cong \eps_l^{n(1-n)}$ mod $l$, which appears not to be accessible by the methods of this paper. I thank the anonymous referee for pointing out the error. A forthcoming version of the manuscript \cite{blggt} will use rather different methods to prove potential automorphy theorems which are stronger than those proved in this paper and which include the result originally claimed in this paper as a special case. (A version of this paper has already been circulated proving this analogue of this stronger result for totally real fields; the new manuscript will extend the result to CM fields.) 

It is perhaps also worth remarking that in almost every application, the condition (7) (in either its old or new forms) is completely harmless anyway. To see why, we first recall that in almost every application, one has a compatible family of representations $r_l$, one for each rational prime $l$, and one generally has a great deal of freedom to choose $l$ to have whatever properties are necessary for the rest of the argument. (In particular, this is how one achieves the condition that $l$ split in the field $F^*$ of Theorem \ref{orig-main-theorem}.) In such a situation, one can arrange that $\Q_l$ contains a good supply of roots of unity, and in this case, one use a twisting argument to eliminate condition (7) entirely. The details are briefly sketched in section \ref{sec:twisting}.

{\bf Acknowledgements:} I would like to thank my advisor, Richard Taylor, for suggesting this problem to me and for immeasurable help in all aspects of my work on it. I would also like to thank the anonymous referee for many helpful suggestions, and (as I have mentioned) for pointing out an important oversight.

\section{Geometry}\label{sec:geom}

\subsection{The Dwork family} Our aim in this section is to prove a proposition that allows us to find varieties with prescribed residual representations, and in order to do so we must introduce the Dwork family, within which we will find the varieties we seek. Let $N$ be a positive integer. Fix a base ring $R_0=\Z[\frac{1}{N},\mu_N]$, where $\mu_N$ denotes the $N$th roots of unity. We consider the scheme $Y$:
$$Y\subset \Pj^{N-1} \times \Pj^1$$
over $R_0$ defined by the equations
$$\nu(X_1^{N}+X_2^{N}+\dots+X_N^{N}) = N\lambda X_1 X_2\dots X_N$$
(using $(X_1:\dots:X_N)$ and $(\nu:\lambda)$ as coordinates on $\Pj^{N-1}$ and $\Pj^1$ respectively.) We consider $Y$ as a family of schemes over $\Pj^1$ by projection to the second factor. We will label points on this $\Pj^1$ using the affine coordinate $t=\lambda/\nu$, and will write $Y_t$ for the fiber of $Y$ above $t$. (The notation broadly follows Katz's paper \cite{k}, except that I use $N$ in place of his $n$, $Y$ for his $\bbX$, and the varieties I consider are less general than his---corresponding to the case $W=(1,1,\dots,1)$ and $d=n$ in his notation. In particular, our notation is \emph{not} directly compatible with the notation of \cite{hsbt}.)

There is a natural group acting on this family. Let $\mu_N$ denote the $N$th roots of unity in $R_0$, and let $\Gamma$ denote the $N$ fold power $(\mu_N)^N$. Let $\Gamma_W$ denote the subgroup of $\Gamma$ consisting of all elements $(\zeta_1,\dots \zeta_N)$ with $\prod_{i=1}^N \zeta_i=1$ and let $\Delta$ denote $\mu_N$ embedded diagonally in $\Gamma$. Then the group $\Gamma_W$ acts on $Y$ with the element $(\zeta_1,\dots \zeta_N)$ acting via
$$((X_1:\dots:X_N),t)\mapsto ((\zeta_1 X_1:\dots:\zeta_N X_N),t)$$
The subgroup $\Delta$ acts trivially. 

The family $Y$ is smooth over the open set $U=\Spec R_0[\lambda,\frac{1}{\lambda^N-1}]\subset \Pj^1(R_0)$. We will now construct certain sheaves on $U$. Let $l$ be a prime number which splits in $\Q(\mu_N)$, and assume we have chosen an embedding $\iota$ of $R_0$ into $\Qbar_l$. Let $T^{(l)}_0=U\times_{\Spec R_0}\Spec R_0[1/l]$, and form lisse sheaves
\begin{align}
\cF^i_l &:= R^i\pi_* \Q_l		\\
\cF^i[l] &:= R^i\pi_* \Z/l\Z
\end{align}
on $T^{(l)}_0$. (We will suppress  the superscript $(l)$ where it is clear from context.) Similarly, let $M$ be an integer, $T^{(M)}_0=U\times_{\Spec R_0}\Spec R_0[1/M]$ (=$\Spec R_0[\lambda,\frac{1}{\lambda^N-1},1/M]$), and define a lisse sheaf $\cF^i[M] := R^i\pi_* \Z/M\Z$ on $T^{(M)}_0$.

We are interested particularly in the sheaf $\cF_l^{N-2}|_{T_0^{(M)}}$. Form now on, we assume that $N$ is odd; for such an $N$, we write $\Prim_l^{N-2}$ for this $\cF_l^{N-2}|_{T_0^{(M)}}$. (For more general $N$, $\Prim_l^{N-2}$ could be a little more complicated---we might have to take the annihilator under cup product of some cohomology class coming from a power of the hyperplane class from the ambient $\Pj$. But for us this simple definition suffices.)
As has been remarked, $\Gamma_W/\Delta$ acts on our family, and so acts on the sheaf $\Prim_l^{N-2}$ we have just defined: thus we can decompose $\Prim_l^{N-2}$ into eigensheaves according to the characters of the group $\Gamma_W/\Delta$. Note that the coefficient ring of these sheaves will still be $\Q_l$, since $l$ was chosen to split in $\Q(\mu_N)$.

The character group of $\Gamma$ is $(\Z/N\Z)^N$; that of $\Gamma_W$ is $(\Z/N\Z)^N / \langle W\rangle$ where we write $W$ for the element $(1,1,\dots,1)$; and the character group of $\Gamma_W/\Delta$ is $(\Z/N\Z)_0^N / \langle W\rangle$ where we write $(\Z/N\Z)_0^N$ for $\{(v_1,\dots,v_N)\in (\Z/N\Z)^N|\sum_i v_i = 0\}$. Thus the eigensheaves are labeled by elements of $(\Z/N\Z)_0^N / \langle W\rangle$: we will write such an element as $(v_1,\dots,v_N)$ mod $W$ or simply as $[(v_1,\dots,v_N)]$, and shall write $\Prim^{N-2}_{l,[(v_1,\dots,v_N)]}$ for the piece of $\Prim_l^{N-2}$ where $\Gamma_W/\Delta$ acts via $[(v_1,\dots,v_N)]$. Note that this labeling \emph{depends on the choice of embedding} $\iota:R_0\ra\Q_l$, since it requires us to have a preferred identification of the roots of unity in the coefficient ring $\Q_l$ of the cohomology with the roots of unity in $R_0$.

We now are in a position to single out the particular piece of the cohomology with which we will work. From now on we will assume that we have another positive even integer $n$ in mind, with $N\geq n+5$. (This will be the dimension of the Galois representation which we will be working with in the end.) We will write $k$ for $n/2$, and (assuming for the moment that $n>2$) we will set 
\begin{align*}
v = (0,\dots,&0, 2, k+1,k+3,k+4, \dots, N-k-2,N-1)
\end{align*}
where we include every number once, except we \emph{omit} the ranges 
$3,\dots,k$, and $N-k-1,\dots,N-2$, and the singletons 1 and $k+2$, and where the number of 0s at the beginning is $n+1$, calculated to ensure that there are $N$ numbers in total. Note that these numbers add up to 0 mod $N$. Note also that the ranges above `make sense' as long as $N \geq n+5$. (For instance, if $n=4, N=9$, we take $v=(0,0,0,0,0,2,3,5,8)$.) Note finally that this choice of $v$ has the property that $-v$ is never a permutation of $v$.

On the other hand, if $n=2$, we take $v=(0,0,0,2,3,\dots,N-2)$, where we include every number once except we omit 1 and $N-1$ and include 0 three times. We see again that these numbers add up to 0 mod $N$, but that for $n=2$, $-v$ is in fact a permutation of $v$.

We will work with the piece $\Prim^{N-2}_{l,[v]}$ under which $\Gamma_W/\Delta$ acts via this $v$ mod $W$.\footnote{This is the point of where we part company from \cite{hsbt}; they work with $N=n+1$ and the piece $[(0,0,\dots,0)]$.} I will often write $\Prim_{l}$ for this sheaf, with the remaining data being understood. I will write $\Prim[l]$ for the corresponding sheaf constructed from $\cF^i[l]$, and $\Prim[M]$ from the corresponding sheaf constructed from $\cF^i[M]$.

\begin{proposition} \label{geom-lemma}
Let $F$ will denote to a CM field containing 
$R_0$, $v$ to a place of $F$.  There is a natural morphism 
$\Spec F[\lambda,\frac{1}{\lambda^N-1}]=\Spec F\times_{\Spec R_0}T^{(M)}_0\to\Spec T^{(M)}_0$
and by pulling back along it, we consider $\Prim_{l}$ and $\Prim[l]$ as sheaves on 
$\Spec F[\lambda,\frac{1}{\lambda^N-1}]$.

Then we have the following facts about the varieties $Y_t$ and the sheaves $\Prim[l]$, 
$\Prim_l$ and $\cF^i[l]$. (Recall that we are assuming $l\equiv 1$ mod $N$ 
throughout.)
\begin{enumerate}
\item \label{geom-lemma-good-red}
If $t\in T^{(l)}_0(F)$ and $\fq$ is a place of $F$ such that $v_\fq(1-t^N)=0$, then $Y_t$ has good reduction at $\fq$.
\item \label{geom-lemma-dual}
Let us write $F^+$ for the maximal totally real subfield of $F$ and $c$ for the 
nontrivial element of $\Gal(F/F^+)$, which we can think of as a field automorphism of $F$. This $c$
then induces an automorphism of the ring $F[\lambda,\frac{1}{\lambda^N-1}]$ over which $\Prim_l$
is defined. We will write $\Prim_l^c$ (resp. $\Prim[l]^c$) for the pull-back of $\Prim_l$ 
(resp. $\Prim[l]$) along this map. Now suppose that $t\in T^{(l)}_0(F)$. Then the Galois representation 
$$\Prim_{l,t}:\Gal(\Fbar/F) \ra GL_n(\Q_l)$$  
satisfies $\Prim_{l,t}^c \cong \Prim_{l,t}^\vee \eps_l^{2-N}$. Similarly $\Prim[l]_{t}^c \cong \Prim[l]_{t}^\vee \eps_l^{2-N}$, and indeed these isomorphisms patch for different $t$ to give a sheaf isomorphism.
\item \label{geom-lemma-ht}
The sheaf $\Prim_{l}$ has rank $n$. There is a tuple $\vec{h}=(h(\sigma))_{\sigma\in\Hom(F,\Qbar_l)}$, such that the  Hodge-Tate numbers of $\Prim_{l}$ at the embedding $\sigma$ are $\{h(\sigma),h(\sigma)+1,\dots,h(\sigma)+n-1\}$.

\item \label{geom-lemma-inertia}
Let $\vec{h}$ continue to denote the tuple defined in the previous part. Suppose $\fw|l$, and let $\sigma\in\Hom(F,\Qbar_l)$ denote any of the corresponding embeddings. Then $\Prim_{l,0}|_{I_\fw}\cong \eps_l^{-h(\sigma)}\oplus\eps^{-h(\sigma)-1}_l\oplus\dots\oplus\eps_l^{1-h(\sigma)-n}$, and $\Prim[l]_0|_{I_\fw}\cong\eps_l^{-h(\sigma)}\oplus\eps^{-h(\sigma)-1}_l\oplus\dots\oplus\eps_l^{1-h(\sigma)-n}$
\item \label{geom-lemma-steinberg}
Let $\fq$ be a prime of $F$ above a rational prime which does not divide $N$. If $\lambda_\fq\in T^{(l)}_0(F_\fq)$ has $v_\fq(\lambda_\fq) < 0$, then $(\Prim_{l,\lambda_\fq})\ssm$ is unramified, and $(\Prim_{l,\lambda_\fq})\ssm(\Frob_\fq)$ has eigenvalues $\{\alpha,\alpha\#k(\fq),\alpha(\#k(\fq))^2,\dots,\alpha(\#k(\fq))^{n-1}\}$ for some $\alpha$.
\item \label{geom-lemma-unram-mod}
Let $\fq$ be a prime of $F$ above a rational prime which does not divide $N$. If $\lambda_\fq\in T^{(l)}_0(F_\fq)$ has $v_\fq(\lambda_\fq) < 0$ and $l|v_\fq(\lambda_\fq)$, then $(\Prim[l]_{\lambda_\fq})$ is unramified (even without semisimplification).
\item \label{geom-lemma-monod}
The monodromy of $\Prim_{l}$ maps into, and is Zariski dense in, $\SL_n$.
\end{enumerate} 
\end{proposition}
\begin{proof} Point (1) is trivial. Point (2) comes from the fact that there is a perfect Poincare duality pairing between $\Prim^{N-2}_{l,[v]}$ (=$\Prim_l$) and $\Prim^{N-2}_{l,[-v]}$ towards $\Qbar_l(2-N)$, and the fact that we can identify $\Prim^{N-2}_{l,[-v]}$ as the complex conjugate of $\Prim^{N-2}_{l,[v]}$. (And then a similar argument for $\Prim[l]_t$.)

It will prove useful to skip over points (3) and (4) and return to them later. To begin our analysis of points (5) and (6), let us note that it suffices, by an argument identical to that used to prove Lemma 1.15 of \cite{hsbt}, to establish that for $\lambda_\fq$ of the form given the monodromy of $\Prim_l$ around infinity is generated by a unipotent matrix with minimal polynomial $(X-1)^n$. Now we will apply Lemma 10.1 of \cite{k}. It is clear from the definition of $v$ we gave that point (4) of the equivalent conditions given in this lemma is true (viz, that the value 0 occurs more than once and no other value does); whence we can deduce the equivalent condition (2), which is the unipotence we need. 

Next, we move to establish point (7). We apply Lemma 10.3 of \cite{k}. When $n>2$, we saw that the $v$ we chose did not have $-v$ a permutation of $v$. Thus we are in case (1) of \cite[Lemma 10.3]{k}, and the geometric monodromy is dense in $\SL_n$ , since we took $N$ (which corresponds to $n$ in Katz's notation) to be odd. On the other hand, if $n=2$ then $v$ is a permutation of $-v$ and again using the fact that $N$ is even we are in case (2) of  \cite[Lemma 10.3]{k}, and the geometric monodromy is dense in $\Sp_2$. But $\Sp_2=\SL_2$, so in this case again the geometric monodromy is dense in $\SL_n$.
This establishes point (7) of the present proposition.

We now move on to establish point (3). First, we will apply Lemma 3.1 of \cite{k}, which gives a recipe for computing the ranks of the eigensheaves of $\Prim_l^{N-2}$, and another recipe for computing the Hodge-Tate numbers. We will apply the recipe for the ranks. We are asked to consider the coset of elements of $(\Z/N\Z)_0^N$ representing $v$, and in particular, those elements of the coset which are \emph{totally nonzero}; that is, contain no 0s. The translate $v-(y,y,\dots,y)$ will be totally nonzero iff $y$ does not occur in $v$; as discussed above, our $v$ omits precisely $n$ congruence classes mod $N$, hence there are $n$ totally nonzero representatives. The rank equals the number of totally nonzero representatives, which will therefore be $n$.

Now, we apply Lemma 10.4 of \cite{k}, which tells us that (when the equivalent conditions of Lemma 10.1 of \cite{k} hold, as they do for us) the Hodge-Tate weights form an unbroken string of ones; that is, the Hodge-Tate numbers are of exactly the form we require, where we define $h(\sigma)$ to be the smallest Hodge-Tate number at the embedding $\sigma$.

We next prove point (4) in the special case where $F=\Q(\mu_N)$. We first observe that the group $\Gamma/\Delta$ (rather than just $\Gamma_W/\Delta$) actually acts on $H(Y_0)$, allowing us to decompose $\Prim_{l,0}$ further into eigensheaves for $\Gamma/\Gamma_W$. Proposition I.7.4 of \cite{dmos} tells us that these eigensheaves are all one dimensional, and since $l$ is chosen to split in $\Q(\mu_N)$, this tells us that $\Prim_{l,0}$ is a direct sum of characters, which are crystalline with Hodge-Tate numbers $\{h,h+1,\dots,h+n-1\}$ by point (3). This establishes the first part (since $l$ splits in $F=\Q(\mu_N)$, there is a unique embedding $\sigma$ corresponding to $\fw$ and a crystalline character of $I_\fw$ is a power of the cyclotomic character, and we can read off which one by examining its Hodge-Tate number at $\sigma$). The second part then follows, since as $l\equiv 1$ mod $N$, we have $l>N>n$ and the characters $\eps_l^{-1},\eps_l^{-2},\dots,\eps_l^{1-n}$ have distinct reductions mod $l$.

Finally, we deduce the general case of point (4). Letting $F$ now be arbitrary, and $\fw|l$ an arbitrary prime of $F$, we let $\fw'$ be the prime of $\Q(\mu_N)$ below $\fw$. Now restricting both sides of $\Prim_{l,0}|_{I_{\fw'}}\cong \eps_l^{-h(\sigma)}\oplus\eps^{-h(\sigma)-1}_l\oplus\dots\oplus\eps_l^{1-h(\sigma)-n}$ to $G_F$ gives the result concerning $\Prim_{l,0}|_{I_{\fw}}$. A similar argument works for $\Prim[l]_{0}|_{I_{\fw}}$.
\end{proof}

Now, we have the following Corollary.
\begin{corollary} \label{monodromy-mod-prop} There is a constant $C(n,N)$ such that if $M$ is an integer divisible only by primes $p>C(n,N)$ and if $t\in T_0^{(M)}$ then the map
$$\pi_1(T_0^{(M)},t) \ra \GL(\Prim[M]_t)$$
surjects onto $\SL(\Prim[M]_t)$. (We may, and shall, then additionally assume that $C(n,N)>n$.)
\end{corollary}
\begin{proof}
First note that Lemma \ref{geom-lemma}, part \ref{geom-lemma-monod} tells us that $\pi_1(T_0^{(M)},t) \ra \GL(\Prim_{l,t})$ maps into $\SL(\Prim_{l,t})$, so we certainly have that the map $\pi_1(T_0^{(M)},t) \ra \GL(\Prim[l]_t)$ maps into $\SL(\Prim[l]_t)$ for each prime $l$; and hence the map $\pi_1(T_0^{(M)},t) \ra \GL(\Prim[M]_t)$ factors through $\SL(\Prim[M]_t)$. Thus the key point is that the map into $\SL(\Prim[M]_t)$ is in fact surjective.

We will show first that  this map onto $\SL(\Prim[M]_t)$ is a surjection in the case where $M$ is a prime power, $M=l^a$, for $l$ greater than some bound $C(n,N)$.

To make our argument, we must introduce transcendental versions of some of the objects we have been considering. In particular, let us define $T_{0,\C}=U\times_{\Spec R_0} \Spec \C$ (a scheme over $\Spec \C$), and define $T_0^{(an)}$ to be the topological space associated with this scheme using the standard complex topology. Let us similarly put $Y_\C = Y\times_U T_{0,\C}$ and let $Y^{(an)}$ denote the topological space associated to the scheme $Y_\C$ by the complex topology. Let us write $\Prim_l^{N-2,\C}$ for the base-change of the sheaf $\Prim_l^{N-2,\C}$ to $T_{0,\C}$. Finally, let $\pi^{(an)}:Y^{(an)}\to T_0^{(an)}$ be the natural map, and let us write (for each $i$), $\cF^i_{(an)}=R^i\pi^{(an)}_* \Q$, a transcendentally defined local system with $\Q$ coefficients on $T_0^{(an)}$. Let $\Prim^{N-2}_{(an)}:=\cF_{(an)}^{N-2}$. 

We now apply standard comparison theorems to relate these objects. Firstly, we know that we have an isomorphism between $\pi_1(T_0^{(M)},t)$ and $\pi_1(T_{0,\C},t)$ (the algebraic fundamental group of the scheme $T_{0,\C}$ over $\C$) and between $\pi_1(T_{0,\C},t)$ and $(\pi_{1,\mathrm{top}}(T_{0}^{(an)},t))^\vee$ (the profinite completion of the topological fundamental group of the topological space $T_{0}^{(an)}$). Second, for $t\in T_{0,\C}$ and a prime $l$,  we can identify the fiber  $\Prim_{l,t}^{N-2,\C}$ with $\Prim^{N-2}_{(an),t}\otimes_{\Q}\Q_l$, while for $t$ a geometric point in $ T_{0}^{(M)}$ we can identify $\Prim_{l,t}^{N-2}$ and $\Prim_{l,t}^{N-2,\C}$; so that for $t$ such a geometric point we can identify $\Prim^{N-2}_{(an),t}\otimes_{\Q}\Q_l$ with $\Prim_{l,t}^{N-2,\C}$. In particular, since $\Prim^{N-2}$ had rank $n$, so do $\Prim_{l,t}^{N-2,\C}$ and $\Prim^{N-2}_{(an)}$. And finally, with $t$ still a geometric point in $T_{0}^{(M)}$, these identifications are compatible in the sense that we have a commutative diagram for each prime $l$:
$$\xymatrix{
\pi_{1,\mathrm{top}}(T_0^{(an)})\ar@{^{(}->}[d]\ar[r] 	&
				\GL(\Prim^{N-2}_{(an),t})\ar@{->}[rr]^{\sim}\ar@{^{(}->}[d] & &
								\GL_n(\Q)\ar@{^{(}->}[d]^{i}		\\
(\pi_{1,\mathrm{top}}(T_0^{(an)}))^\vee\ar[d]^{\sim}		&
				\GL(\Prim^{N-2}_{(an),t}\otimes_{\Q}\Q_l)\ar@{->}[rr]^{\sim}\ar[d]^{\sim} &&
								\GL_n(\Q_l)\ar@{=}[d]				\\
\pi_1(T_{0,\C},t)\ar[d]^{\sim}\ar[r]			&
				\GL(\Prim^{N-2,\C}_{l,t})\ar[d]^{\sim}\ar@{->}[rr]^{\sim}		&&
								\GL_n(\Q_l)\ar@{=}[d]	\\
\pi_1(T_0^{(M)},t)\ar[r]			&
				\GL(\Prim^{N-2,\C}_{l,t})\ar@{->}[rr]^{\sim}		&&
								\GL_n(\Q_l)
}$$
On the other hand, we know that the bottom line in fact maps into $\SL_n(\Q_l)$, and this tells us that we can replace $\GL$ with $\SL$ everywhere in the diagram, and we will from now on consider this substitution to have been made. Let us write $\Gamma$ for the image of the map along the top of the diagram in $\SL_n(\Q)$. Since, as a topological space, $T_0^{(an)}$ is a sphere with $N+1$ points removed (the Riemann sphere having removed the point $\infty$ and the $N$th roots of 1), $\pi_{1,(top)}(T_0^{(an)})$ is a free group with $N$ generators, say $\gamma_1,\dots \gamma_{N}$. Thus $\Gamma$ is a finitely generated subgroup of $\SL_n(\Q)$. 

We next claim that $\Gamma$ is Zariski dense in $\SL_n(\Q)$. To see this, let us let $\Gamma'$ denote the image of the map along the bottom of the diagram (a subset of $\SL_n(\Q_l)$). (Given the vertical isomorphisms between the bottom three lines in the diagram, we can equally well think of this as the image of the map along the third line of the diagram or the image of the horizontal map in the second line.) We know that $\Gamma'$ is Zariski dense in $\SL_n(\Q_l)$ by Lemma \ref{geom-lemma}, part \ref{geom-lemma-monod}, and we also see that $\Gamma'$ is contained in the $l$-adic closure of the image of $\Gamma$ under the map $i$. We thus see that the image of $\Gamma$ in $\GL_n(\Q_l)$ is Zariski dense. It follows that $\Gamma$ is Zariski dense in $\GL_n(\Q)$.

Then \cite[Theorem 7.5]{mvw} tells us that if we are given:
\begin{itemize}
\item $G$, a connected, absolutely simple algebraic group of adjoint type over $\Q$,
\item $\tilde{G}$ its simply connected cover (which will be an `almost simple' algebraic group---that is, one with no connected normal subgroups), and
\item $\Gamma$, a finitely generated Zariski dense subgroup of $\tilde{G}(\Q)$,
\end{itemize}
then the image of $\Gamma$ in $\tilde{G}(\Q_l)$ is $l$-adically dense for all but finitely many $l$. Applying this with $G$ the group $\PGL$ (and hence $\tilde{G}$ being $\SL$), we see that for almost all $l$, $\Gamma$, and hence the image of $\pi_{1,\mathrm{top}}$, is $l$-adically dense in $\SL_n(\Q_l)$. In particular, choosing $C(n,N)$ to be the largest $l$ for which the image is not $l$-adically dense, we have that for $M=l^a$ with $l>C(n,N)$, the map $\pi_{1,\mathrm{top}}\ra \SL(\Prim[M]_t)$ is surjective, and hence certainly $\pi_1(T_0^{(M)},t) \ra \SL(\Prim[M]_t)$ is surjective.

\smallskip

Thus we have seen that the map $\pi_1(T_0^{(M)},t) \ra \SL(\Prim[M]_t)$ is surjective where $M=l^a$ is a prime power and $l>C(n,N)$. We shall now show that that the map $\pi_1(T_0^{(M)},t) \ra \SL(\Prim[M]_t)$ remains surjective for $M=p_1^{a_1}\dots p_r^{a_r}>C(n,N)$ an arbitrary positive integer, which will complete the proof of the corollary. Note that we can certainly see that in this case each $p_i>C(n,N)$.

Let us write $\phi$ for the map $\pi_1(T_0^{(M)},t) \ra \SL(\Prim[M]_t) \isoto \prod_i  \SL(\Prim[p_i^{a_i}]_t)\isoto \prod_i\SL_n(\Z/p_i^{a_i}\Z)$; we wish to show this map is surjective. Let $\Gamma = \phi(\pi_1(T_0^{(M)},t)) \subset \prod_i\SL_n(\Z/p_i^{a_i}\Z)$. Now, \cite[Lemma 8.4]{mvw} tells us that if we are given:
\begin{itemize}
\item $G$, a connected, absolutely simple algebraic group of adjoint type over $\Q$,
\item $\tilde{G}$ its simply connected cover (which will be an `almost simple' algebraic group---it has no connected normal subgroups), 
\item $p_1,\dots,p_r$ a collection of rational primes, 
\item $a_1,\dots, a_r$ a collection of non-negative integers, and 
\item $\Gamma$ a subgroup of $\prod_i \tilde{G}(\Z/p_i^{a_i}\Z)$
\end{itemize}
then \emph{if} the projection of $\Gamma$ onto $\tilde{G}(\Z/p_i^{a_i}\Z)$ for each $i$ is all of $\tilde{G}(\Z/p_i^{a_i}\Z)$, \emph{then} in fact $\Gamma=\prod_i \tilde{G}(\Z/p_i^{a_i}\Z)$. We can apply this lemma, with $G=\PGL_n$ (and hence $\tilde{G}=\SL_n$), and with the $p_i$, the $a_i$, and $\Gamma$ as in their present contexts, to see that \emph{if} the projection of $\Gamma$ onto $\SL_n(\Z/p_i^{a_i}\Z)$ is all of $\SL_n(\Z/p_i^{a_i}\Z)$ for each $i$, \emph{then} we will in fact have $\Gamma = \prod_i\SL_n(\Z/p_i^{a_i}\Z)$, telling us that $\phi$ is surjective, as required. Hence it suffices to show that the projection of $\Gamma$ onto $\SL_n(\Z/p_i^{a_i}\Z)$ is all of $\SL_n(\Z/p_i^{a_i}\Z)$ for each $i$. But this projection is the image of the natural map $\pi_1(T_0^{(M)},t) \isoto\pi_1(T_0^{(l)},t) \ra \SL(\Prim[l]_t) \isoto \SL_n(\Z/p_i^{a_i}\Z)$, and we saw above that $\pi_1(T_0^{(l)},t) \ra \SL(\Prim[l]_t)$ is surjective. (This uses the fact that each $p_i>C(n,N)$).
\end{proof}

\section{Realizing residual representations}\label{sec:moret-bailly}

In this section, our ultimate aim is to use the theorem of Moret-Bailly to prove a result allowing us to realize (restrictions of) prescribed residual representations as the cohomology of varieties in the Dwork family. Before we do so, we must prove certain technical results which will be necessary to us in this goal.

\subsection{Hypergeometric sheaves}
The first of these technical results which we will are certain calculations concerning the determinant $\det\Prim_l$, but before we can make these calculations, we must review some material concerning certain `hypergeometric' sheaves studied by Katz, since our main tool in studying the determinant will be results of Katz which relate it to those hypergeometric sheaves. We will carry out this review in this section, which may therefore be skimmed or skipped entirely by readers already familiar with this material.

There are in fact two distinct kinds of hypergeometric sheaf which will be of importance to us in the sequel. On the one hand, and of primary importance, there are the \emph{canonical hypergeometric sheaves}, which are \'etale  sheaves on scheme $\G_m-\{1\}$ over $R_0$; in particular, they should be thought of as \emph{global} objects (since $\Frac(R_0)$ is a number field). On the other, there is another kind of hypergeometric sheaf which we will call \emph{traditional hypergeometric sheaves}. (These are the original hypergeometric sheaves studied by Katz in, say, \cite{k-book}. There, they are simply called `hypergeometric sheaves', since the canonical variant had not been invented at the time---but for us it will prove useful to attach the word `traditional' to them, to distinguish them from the canonical hypergeometric sheaves which will be of greater importance here.) We will discuss each in turn.

\subsubsection{Canonical hypergeometric sheaves.}
Let us write $B$ for the scheme $\G_m-\{1\}$ over $\Spec R_0$, fix a rational prime $l$, and suppose we are given multisets\footnote{A \emph{multiset} is a notion similar to a set, except we keep track of multiplicity of membership. For instance, $\{1,2,2,3\}$ and $\{1,2,3\}$ are the same as sets, but would be considered different as multisets. (We will write multisets using the same `$\{\}$' notation used for sets, but it should always be clear from the context when we mean for this notation to denote a multiset and when a set.)  Formally, a multiset can be thought of as a function from a set to the positive integers, where the positive integer associated to an element is its multiplicity. For full details, see \cite[pp1026--1039]{gs}.} $S_\chi$ and $S_\rho$ of characters $\mu_N\ra\mu_N$, each of size $k$. (A quick point of convention: although we have been writing such characters `additively' up until now, as elements of $\Z/N\Z$, it will be convenient in this section to switch to multiplicative notation to match the notation used by Katz. Thus, for instance, 1 will denote the trivial character.) 

Given this data, we will define (following Katz\footnote{Note that the construction we are about to present is the construction of the global object $\Hcan$ defined on \cite[p11]{k}, and \emph{not} the local object, also called $\Hcan$, defined on \cite[p11]{k}, about which we will have more to say presently.}) an object called $\Hcan(S_\chi,S_\rho)$, which is a rank $k$ sheaf on $B$ with $\Q_l$ coefficients. The definition proceeds in three stages:
\begin{definition}\label{def:can-hyp}
\begin{enumerate}
\item Suppose that $\chi,\rho$ are characters $\mu_N\ra\mu_N$. If $\mathcal{P}$ is a maximal ideal of $R_0$, then the finite field $\bfk=R_0/\mathcal{P}$ has all $N$th roots of unity, and we can view $\chi$ and $\rho$ as $\mu_N(R_0)$-valued characters of $\bfk^\times$ by composing with the surjective map $p:\bfk^\times\ra\mu_N(\bfk)\isoto\mu_N(R_0)$ obtained by raising to the $\#\bfk^\times/N$th power. Then, by \cite{we-js}, attaching to each maximal ideal $\mathcal{P}$ of $R_0$ the (negative) Jacobi sum $-\sum_{x\in (R_0/\mathcal{P})^\times} (\chi\circ p)(x)((\rho/\chi)\circ p)(1-x)$ defines a grossencharacter, and hence by \cite[Chapter 2]{serre-alr} a $\overline{\Q}_l$-valued character, of $\pi_1(\Spec(R_0[1/l]))$. We will write $\Lambda_{\chi,\rho/\chi}$ for this character. (Note that the interpretation of $\Lambda_{\chi,\rho/\chi}$ as a $\Q_l$ valued character depends on the choice of an embedding $R_0\into \Q_l$, and we will find it convenient to make the same choice as was used to label the pieces of the cohomology of the Dwork family in \S2.) By composing with the natural map $\pi_1(B)\rightarrow\pi_1(\Spec(R_0[1/l]))$, we can also consider $\Lambda_{\chi,\rho/\chi}$ to be a character $\pi_1(B)\rightarrow\overline{Q}_l^\times$.
\item Suppose that $S_\chi$ and $S_\rho$ are both singleton multisets, so that $S_\chi=\{\chi\}$ and $S_\rho=\{\rho\}$ say, and we have a character $\Lambda_{\chi,\rho/\chi}$ as in the previous part attached to $\chi$ and $\rho$. We can also form the Kummer\footnote{The \emph{Kummer sheaf} is defined as follows. Let $[N]$ denote the $N$th power map. This exhibits $B$ as a finite \'etale covering of itself, with fiber $\mu_N(R_0)$, and moreover we see that the action of $\pi(B)$ on the fiber $\mu_N(R_0)$ has the following property: each element of $\gamma\in\pi(B)$ acts by multiplication by some $x(\gamma)\in\mu_N(R_0)^\times$. The map $x:\gamma\mapsto x(\gamma)$ is clearly a homomorphism, and given a character $\chi:\mu_N(R_0)\to\mu_N(\Q_l)$, $\chi\circ x$ determines a homomorphism $\pi_1(B)\to\Q_l^\times$, and hence a lisse sheaf $\cL_{\chi(x)}$ on $B$; this is the Kummer sheaf. We can also view it as a sheaf on $\A^1$ by extension by zero. The related sheaf $\cL_{\chi(1-x)}$ is the pullback of $\cL_{\chi(x)}$ on $\A^1$ along $x\mapsto 1-x$.} sheaves $\cL_{\chi(x)}$ and $\cL_{(\rho/\chi)(1-x)}$ on $B$. We then define $\Hcan(\{\chi\},\{\rho\})$ by putting
$$\Hcan(\{\chi\},\{\rho\}) = \cL_{\chi(x)} \otimes \cL_{(\rho/\chi)(1-x)} \otimes (1/\Lambda_{\chi,\rho/\chi}).$$
\item Now suppose $S_\chi$ and $S_\rho$ are general multisets of characters $\mu_N\ra\mu_N$, each of size $k$. Let us put $S_\chi=\{\chi_1,\dots,\chi_k\}$ and $S_\rho=\{\rho_1,\dots,\rho_k\}$. We can form the sheaves $\Hcan(\{\chi_1\},\{\rho_1\})$, $\Hcan(\{\chi_2\},\{\rho_2\})$ and so on according to the definition in the previous part. These sheaves then give rise to elements of the derived category of sheaves on $B$, which we will also refer to (by abuse of notation) as $\Hcan(\{\chi_1\},\{\rho_1\})$, $\Hcan(\{\chi_2\},\{\rho_2\})$ and so on. We can then apply the `shift' operator [1] to these elements of the derived category, and can finally form the $!$ multiplicative convolution:
$$\Hcan(\{\chi_1\},\{\rho_1\})[1] \star_! \Hcan(\{\chi_2\},\{\rho_2\})[1] \star_!\dots \star_! \Hcan(\{\chi_k\},\{\rho_k\})[1].$$
(See, for instance, \cite[\S8.1.8]{k-book} for the definition of convolution of objects of the derived category of sheaves.) It is then the case that this element of the derived category is in fact of the form $\cF[1]$ for some sheaf\footnote{This sheaf is not to be confused with the sheaf $\cF^i_l$ above.} $\cF$, and $\cF$ does not in fact depend (up to isomorphism) on the ordering on the imposed on the $\chi_i$ and $\rho_i$ and used to define the multiplicative convolution above (see \cite[p11, last paragraph]{k} for both these assertions).

We then define $\Hcan(S_\chi,S_\rho)$ to be this sheaf $\cF$, and call $\Hcan$ a \emph{canonical hypergeometric sheaf}.
\end{enumerate}
\end{definition}

The reason these sheaves are of interest to us is a certain result of Katz relating them to the cohomology sheaves $\Prim_l^{n-2}$ which we have been studying. Before we can state this result we need a definition
\begin{definition} Given two multisets $A$, $B$, which may possibly have some elements in common, write $\mathbf{Cancel}(A,B)$ for the pair of multisets $(A',B')$ defined uniquely by the following properties:
\begin{itemize}
\item The elements of $A'$ are precisely those elements of $A$ which either do not occur in $B$, or which occur with greater multiplicity in $A$ than they do in $B$. The multiplicity by which such an $x$ occurs in $A'$ is 
$$(\text{multiplicity with which $x$ occurs in $A$})-(\text{multiplicity with which $x$ occurs in $B$})$$
where the second term is taken to be 0 if $x$ is not in $B$.
\item The same holds where we replace every $A$ with $B$ and $A'$ with $B'$ and vice virca.
\end{itemize}
\end{definition}

The following is the main Theorem of \cite{k}.
\begin{theorem}[Katz] Denote by $j_1:T^{(l)}_0\into \A^1/\Spec R_0[1/l]$ and $j_2:B\into\A^1/\Spec R_0[1/l]$ be the natural inclusions, and let $[N]:B\ra B$ denote the $N$th power map. Recall that we introduced $v$, an element of $(\Z/N\Z)^N$ which we thought of as an $N$-tuple of characters $\mu_N\to\mu_N$ in \S2.1 above. Let $S_\rho(-v)$ denote the multiset given by taking the $N$ characters occurring in the tuple we get by negating $v$, disregarding their order. Let $S_\chi(v)$ denote the set of all $N$ characters $\mu_N\ra\mu_N$.

Then there is a continuous character $\Lambda_v:\pi_1(\Spec(R_0[1/l]))\to \overline{Q}_l^\times$ and an isomorphism of sheaves on $\A^1/\Spec R_0[1/l]:$
$$j_{1,*} \Prim^{n-2}_l \cong j_{2,*} [N]^* \Hcan(\mathbf{Cancel}(S_\chi(v),S_\rho(-v)))\otimes \Lambda_v.$$
\end{theorem}
\begin{proof} This is essentially \cite[Theorem 5.3]{k}, although we must do a little work to unravel the notation. Specifically, \cite[Theorem 5.3]{k} tells us that we can find a $\Lambda_{V,W}$ satisfying
$$j_{1,*} \Prim^{n-2}_l \cong j_{2,*} [d]^* \cH_{V,W}\otimes \Lambda_{V,W}$$
where $\cH_{V,W}$ is a certain sheaf which Katz has introduced earlier in the discussion and Katz's $d$ is the same as our $N$ (we will spell this out a little more below). Thus we see that we will be done (taking our $\Lambda_v$ as Katz's $\Lambda_{V,W}$) if we can unravel the definition of $\cH_{V,W}$ and see that it is in fact simply $\Hcan(S_\chi(v),S_\rho(-v))$.

The definition of $\cH_{V,W}$ occurs at the beginning of \cite[\S5]{k}. Recall that, in comparing notation with Katz, we should be aware that our $v$ is his $V$, that $W\in \Z^N$ for us is a vector of $N$ 1s, that Katz's $d$ (the sum of all the elements of $W$) is just $N$, and that Katz's $d_W$ (defined as the $d$ divided by the lcm of the elements of $W$) is for us $N$ also. In order to define $\cH_{V,W}$, Katz introduces characters $\rho_{v_1},\dots, \rho_{v_N}:\mu_N\to\mu_N$, where $\rho_{v_i}$ is defined to be raising to the $(v_i/d)d_W$'th power. Since for us $d_W=d$, $\rho_{v_i}$ is just the character of raising to the $v_i$th power; and thus it is precisely $v_i$ thought of as a character $\mu_N\to\mu_N$. Katz then introduces a multiset $\List(-v,W)$, which is by definition the multiset we get by taking all $w_i$th roots of $\rho_{(-v)_i}$ for each $i$, and taking the union-with-multiplicity of all these multisets. Since for us every $w_i$ is 1, the collection of all $w_i$th roots of $\rho_{-v_i}$ is just $\{\rho_{(-v)_i}\}$, and taking the union of these we get the multiset $\{\rho_{(-v)_1},\dots,\rho_{(-v)_N}\}$, so since we have seen $\rho_{v_i}$ is just $v_i$ thought of as a character, we see $\List(-v,W)$ is the multiset $\{(-v)_1,\dots,(-v)_N\}$; that is, it is $S_\rho(-v)$ as defined above. Katz also introduces a multiset $\List(\text{all $d$})$, which is by definition the multiset of all characters of order dividing $d$. Since $d=N$ for us, this is just our $S_\chi(v)$. Then, by definition, $\cH_{V,W}=\Hcan(\mathbf{Cancel}(\List(\text{all $d$}),\List(-v,W)))$, but this $=\Hcan(\mathbf{Cancel}(S_\chi(v),S_\rho(-v)))$.
\end{proof}

\begin{corollary} \label{cor: relate prim and hcan}
With the notation as in the theorem, in fact we have 
$$j_{1,*} \Prim^{n-2}_l \cong j_{2,*} [N]^* \Hcan(S'_\chi,S'_\rho))\otimes \Lambda_v$$
where $S'_\chi$ is some multiset of $n$ characters $\mu_N\to\mu_N$, while $S'_\rho$ consists of $n$ copies of the trivial (identically 1) character $\mu_N\to\mu_N$. 
\end{corollary}
\begin{proof}
Examining the choice of $v$ made before Theorem \ref{geom-prop} (and remembering that although at that point we were writing 0 for the trivial character, here we will write 1 since we have switched to multiplicative notation for this section), we see that the tuple $v$ was defined to contain all the characters $\mu_N\to\mu_N$, \emph{except} with some $n$ characters omitted, and $n$ extra copies of the trivial character added to pad the list to length $N$. It is then trivial to calculate that $\mathbf{Cancel}(S_\chi,S_\rho)=(S'_\chi,S'_\rho)$, where $S'_\chi$ is a multiset of $n$ characters (specifically, the ones which $v$ omits) and $S'_\rho$ consists of $n$ copies of the trivial character. This is as required.
\end{proof}

\subsubsection{Traditional hypergeometric sheaves}
We now turn to the traditional hypergeometric sheaves originally studied by Katz. As was mentioned above, these are naturally `local' objects, by contrast to the `global' objects considered in the previous section: for each finite field $\bfk$ which is an $R_0$ algebra, we will define traditional hypergeometric sheaves as sheaves on $B\times_{\Spec R_0} \Spec_\bfk$, where we continue to write $B$ for the scheme $\G_m-\{1\}$ over $\Spec R_0$. To save space, we will abbreviate $B\times_{\Spec R_0} \Spec_\bfk$ as $B_\bfk$.

While it is possible to give a constructive definition of these sheaves akin to the definition of canonical sheaves given in the previous section, the construction is lengthy and unnecessary to us here, so we will give a nonconstructive definition, singling out the required sheaves by giving their trace function and referring to elsewhere for the proof that there \emph{is} a sheaf with this trace function. (The proof we refer to is essentially the construction already mentioned.)
\begin{theorem}[Katz] \label{thm: trad hyp sheaves exist}
Suppose $\bfk$ is a finite field which is an $R_0$ algebra, write $B_\bfk$ for $\G_m-\{1\}/\bfk$, let $\psi:(\bfk,+)\ra\overline{Q}_l^\times$ be a nontrivial additive character, and suppose that we are given multisets $S_\chi$ and $S_\rho$ of characters $\mu_N\ra\mu_N$, each of size $k$. As in Definition \ref{def:can-hyp}, part (1), we will find useful the map $p:\bfk^\times\ra\mu_N(k)\isoto\mu_N(R_0)\isoto \mu_N(\Q_l)$ obtained by first raising to the $\#\bfk^\times/N$th power, then identifying $\mu_N(k)$ and $\mu_N(R_0)$ using the fact $k$ is an $R_0$ algebra, and finally embedding $R_0\into\Q_l$ using the embedding chosen in \S2.)

For $E/\bfk$ a finite extension, write $B_E$ for $V\times_{\Spec \bfk}\Spec E$, and $\Tr_{E/\bfk}$ and $N_{E/\bfk}$ for the norm and trace maps. Let $\psi_E$ denote $\psi\circ\Tr_{E/\bfk}$, $S_{\chi,E}$ denote the multiset $\{\chi\circ N_{E/\bfk}|\chi\in S_{\chi}\}$ and  $S_{\rho,E}$ denote the multiset $\{\rho\circ N_{E/\bfk}|\rho\in S_{\chi}\}$. For $t\in B_E$, let $V(t)$ denote the variety in $(\G_m)^{2k}/E$ with coordinates $(x_1,\dots, x_k, y_1,\dots, y_k)$ cut out by the equation  $\prod_i x_i = t\prod_i y_i$.

Then there exists a unique sheaf $\cH(\psi;S_\chi,S_\rho)$ on $B_\bfk$ such that the trace of Frobenius at $t$ on $\cH(\psi;S_\chi,S_\rho)$ is:
$$(-1)^{2k-1} \sum_{V(t)(E)} \psi_E(\sum_i x_i - \sum_i y_i) \prod_{\chi_E\in S_{\chi,E}} \chi_E(x_i) \prod_{\rho_E\in S_{\rho,E}} \overline{\rho}_E(x_i).$$
\end{theorem}
\begin{proof} That the sheaf is uniquely determined by the traces of Frobenii is a consequence of the Chebotarev density theorem, so the real content of the theorem is that there exists such a sheaf.  To see this, we can appeal to the construction of such a sheaf in \cite{k-book}. Specifically, in \cite[\S8.2.2]{k-book} an element $\mathrm{Hyp}(!,\psi,S_\chi,S_\rho)$ in the derived category of sheaves is introduced, and in \cite[\S8.4.1]{k-book} various properties of this element are developed until one can see it is of the form $\aleph(!,\psi,S_\chi,S_\rho)[1]$ for some sheaf $\aleph(!,\psi,S_\chi,S_\rho)$ (here [1] denotes the derived-category shift operator). The trace function of $\aleph(!,\psi,S_\chi,S_\rho)$ will then be the same as that of $\mathrm{Hyp}(!,\psi,S_\chi,S_\rho)$, which in \cite[\S8.2.7]{k-book} was seen to be exactly the trace function that we ask $\cH(\psi;S_\chi,S_\rho)$ to have. So we can take  $\cH(\psi;S_\chi,S_\rho) := \aleph(!,\psi,S_\chi,S_\rho)$.

(Also see \cite[\S 5, \P 1--2]{k}, where there is a clear statement of the existence of a certain sheaf whose trace function is as above, but no proof.) 
\end{proof}

\begin{definition} \label{thm: trad hyp sheaves defn} We call the sheaf $\cH(\psi;S_\chi,S_\rho)$ of Theorem \ref{thm: trad hyp sheaves exist} a \emph{traditional hypergeometric sheaf}, and we will continue to write $\cH(\psi;S_\chi,S_\rho)$ for it.
\end{definition}

We will need a more results of Katz giving certain properties of these traditional hypergeometric sheaves which will prove important to us. It will be convenient to write $\F$ for the unique finite field of prime order contained in $\bfk$
\begin{proposition}[Katz] \label{Htrad props}
\begin{enumerate} \item $\cH(\psi;S_\chi,S_\rho)$ is pure of weight $2k-1$ and lisse on $B_\bfk-\{1\}$; the monodromy of $\cH(\psi;S_\chi,S_\rho)$ around 1 is a tame pseudoreflection. 
\item Let $\pi:B\times_{\Spec R_0}\Spec R_0[1/l] \ra B_\bfk$ be the natural map (recall that by hypothesis, $\bfk$ is an $R_0$ algebra). Then 
$$\pi^*(\Hcan(S_\chi,S_\rho)) = \cH(\psi;S_\chi,S_\rho) \otimes (1/\phi)$$
where $\phi$ is the unique character of the Galois group of $\bfk$ sending a Frobenius to the following product of Gauss sums
$$\left(\prod_{\chi\in S_\chi} (-g(\psi,\chi))\prod_{\rho\in S_\rho} (-g(\overline{\psi},\overline{\rho}))\right)^{[\bfk:\F]}.$$
\end{enumerate}
\end{proposition}
\begin{proof} 
\begin{enumerate} \item Clear statements of all these facts are found in the first paragraph of \cite[\S4]{k}, although no proofs appear there. For proofs, see \cite[Theorem 8.4.2, (4)]{k-book} (for purity) and \cite[Theorem 8.4.2, (8)]{k-book} (for the remainder).
\item On \cite[p10]{k} a sheaf (called there $\Hcan(\psi;S_\chi,S_\rho)$) on $B_\bfk$ is defined by the formula  $\cH(\psi;S_\chi,S_\rho) \otimes (1/\phi)$ with $\phi$ as above. In \cite[p10, \P3]{k}, it is explained that the pullback of $\Hcan(S_\chi,S_\rho)$ to $B_\bfk$ is $\cH(\psi;S_\chi,S_\rho)$, as required.
\end{enumerate}
\end{proof}

\begin{proposition}[Katz] \label{detformula}
We consider $\det\cH(\psi,S_\chi,S_\rho)$, which will be a one-dimensional sheaf on $B_\bfk$. Let $A$ denote the character of $G_\bfk$ which sends a Frobenius to  
$$\left((\prod_{\chi\in S_\chi} \chi((-1)^{n-1})) q^{n(n-1)/2}\prod_{\chi\in S_\chi,\rho\in S_\rho}(-g(\bar{\psi},\bar{\rho}/\bar{\chi}))
\right)^{[\bfk:\F]}.$$
Then 
$$\det\cH(\psi,S_\chi,S_\rho)=\begin{cases}
A \otimes \cL_{(\prod_{\chi\in S_\chi} \chi)(x)} &\text{(if $\prod_{\chi\in S_\chi}\chi = \prod_{\rho\in S_\rho}\rho$)} \\
A \otimes \cL_{(\prod_{\chi\in S_\chi} \chi)(x)}\otimes \cL_{((\prod_{\rho\in S_\rho} \rho)/(\prod_{\chi\in S_\chi}\chi))(1-x)} &\text{(otherwise)} \\
\end{cases}$$
\end{proposition}
\begin{proof} This is \cite[Theorem 8.12.2, cases 1a, 1b]{k-book}.
\end{proof}

\subsubsection{Some conventions, and a lemma.} \label{sssec: some conventions and a lemma}
We now establish some notational conventions and prove a Lemma. We can consider the sheaf $\Hcan(S_\chi,S_\rho)$ as a representation of $\pi_1(B)$. Since $B$ has a rational point (we will choose, in particular, the point $2^N$), we can consider $\pi_1(B)$ to be $G_{\Q(\mu_N)}\ltimes \pi_1(B\times\Q^{ac})$. The determinant $\det\Hcan(S_\chi, S_\rho)$ will be a character of this group, and any such character will factor through the abelianization $(G_{\Q(\mu_N)}\ltimes \pi_1(B\times\Q^{ac}))^{ab}$. We then claim that:
$$(G_{\Q(\mu_N)}\ltimes \pi_1(B\times\Q^{ac}))^{ab}=(G_{\Q(\mu_N)})^{ab}\times (\pi_1(B\times\Q^{ac})^{ab})_{G_{\Q(\mu_N)}}.$$
(Here $(\pi_1(B\times\Q^{ac})^{ab})_{G_{\Q(\mu_N)}}$ denotes the coinvariants of $\pi_1(B\times\Q^{ac})^{ab}$ with respect to the natural action of $G_{\Q(\mu_N)}$.) This is an example of the general fact that when $G$ and $H$ are groups and we are given an action of $G$ on $H$ and form the semidirect product $G\ltimes H$, then we can write\footnote{
I am grateful to Greg Kuperburg for explaining how to prove this fact in a few sentences. The subgroup $[G\ltimes H,G \ltimes H]$ of $G\ltimes H$ is generated by $[H,H]\cup [G,H] \cup [G,G]$, so we can write
$(G\times H)^{ab}=(G\ltimes H)/\langle [H,H] \cup [G,H] \cup [G,G]\rangle$.
If we apply the relators $[H,H]$ to $G\ltimes H$, we get $G\ltimes H^{ab}$; then if we apply the relators $[G,H]$, we get $G\times (H^{ab})_G$; then finally if we apply $[G,G]$, we get $G^{ab}\times(H^{ab})_G$.} the abelianization of this semidirect product $(G\ltimes H)^{ab}$ as $G^{ab}\times (H^{ab})_G$.

Thus the determinant $\det\Hcan(S_\chi, S_\rho)$, a character of $(G_{\Q(\mu_N)}\ltimes \pi_1(B\times\Q^{ac}))^{ab}$ and hence of $(G_{\Q(\mu_N)})^{ab}\times (\pi_1(B\times\Q^{ac})^{ab})_{G_{\Q(\mu_N)}}$, can be written as a product of a character of $(\pi_1(B\times\Q^{ac})^{ab})_{G_{\Q(\mu_N)}}$ and a character of $(G_{\Q(\mu_N)})^{ab}$. We will write $\det\Hcan(S_\chi, S_\rho)|_{G_{\Q(\mu_N)}}$ for this latter character of $G_{\Q(\mu_N)}$.

Finally, we note that if $S_\chi=\{\chi\}$ and $S_\rho=\{\rho\}$ have size 1, $\Hcan(\{\chi\}, \{\rho\})$ is already a character, which we will call $\lambdacan(\{\chi\}, \{\rho\})$. We can again factor this as a character of $(\pi_1(B\times\Q^{ac})^{ab})_{G_{\Q(\mu_N)}}$ and a character of $(G_{\Q(\mu_N)})^{ab}$, and we will write $\lambdacan_{G_{\Q(\mu_N)}}(\{\chi\}, \{\rho\})$ for the latter character.

\begin{lemma} \label{detHcan lemma}
Suppose that we are given $k$ characters $\chi_1,\dots, \chi_k:\mu_N\to\mu_N$ satisfying $\prod\chi_i=1$. Then we have that:
\begin{align*}
\det\Hcan(\{\chi_1,\dots,\chi_n\},&\{1,\dots,1\})|_{G_{\Q(\mu_N)}} \\&= \lambdacan_{G_{\Q(\mu_N)}}(\{\chi_1\}, \{1\})^{n}\dots\lambdacan_{G_{\Q(\mu_N)}}(\{\chi_n\}, \{1\})^{n} \eps_l^{n(1-n)/2}
\end{align*}
\end{lemma}
\begin{proof} Since both sides are characters which factor through $(G_{\Q(\mu_N)})^{ab}$, it will suffice by Chebotarev to show that they agree on Frobenii. But at a finite place $\cP$ above a rational place $q$, and for $S_\chi$ and $S_\rho$ any pair of multisets of characters of equal size $k$ and with the product of the elements of each multiset 1, we have that
\begin{align*}
(\det&\Hcan(S_\chi,S_\rho)|_{G_{\Q(\mu_N)}})(\Frob_\cP)  \\
 &= \det\Hcan(S_\chi,S_\rho) (\Frob_{\cP,2^N}) \\
 \intertext{(Here $\Frob_{\cP,2^N}$ denotes Frobenius at the point $2^N\in B$; we get the Frobenius at this point rather than any other because we chose to use the point $2^N$ to think of $\pi_1(B)$ as $G_{\Q(\mu_N)}\ltimes \pi_1(B\times\Q^{ac})$.) Then, since $\Frob_{\cP,2}$ is local at $\cP$, we can rewrite this as follows, writing $\bfk$ for $R_0/\cP$  and $\pi: B\times_{\Spec R_0} \Spec R_0[1/l]\to B\times_{\Spec R_0} \Spec \bfk$ for the natural map.}
 &= \det\pi^*(\Hcan(S_\chi,S_\rho)) (\Frob_{\cP,2}) \\
 &= \frac{\det \cH(\psi; S_\chi,S_\rho) (\Frob_{\cP,2})}
       {(\prod_{\chi\in S_\chi} (-g(\psi,\chi))\prod_{\rho\in S_\rho} (-g(\overline{\psi},\overline{\rho})))^{[\bfk:\F]}} \quad \text{(by Prop.~\ref{Htrad props})}\\
 &= \left(\frac{
 (\prod_{\chi\in S_\chi} \chi((-1)^{n-1})) q^{n(n-1)/2}\prod_{\chi\in S_\chi,\rho\in S_\rho}(-g(\bar{\psi},\bar{\rho}/\bar{\chi}))
}
       {\prod_{\chi\in S_\chi} (-g(\psi,\chi))\prod_{\rho\in S_\rho} (-g(\overline{\psi},\overline{\rho})))}\right)^{[\bfk:\F]}  \cL_{(\prod_{\chi\in S_\chi} \chi)(x)} (\Frob_{\cP,2^N})\\
\intertext{(by Prop.~\ref{detformula}, recalling that the product of all the $\chi_i$ and all the $\rho_i$ are both 1)}
 &= \left(\frac{
 q^{n(n-1)/2}\prod_{\chi\in S_\chi,\rho\in S_\rho}(-g(\bar{\psi},\bar{\rho}/\bar{\chi}))
}
       {\prod_{\chi\in S_\chi} (-g(\psi,\chi))\prod_{\rho\in S_\rho} (-g(\overline{\psi},\overline{\rho})))}\right)^{[\bfk:\F]} 
\end{align*}
(since we assume $\prod_{\chi\in S_\chi} \chi$ is trivial, and therefore $\cL_{(\prod_{\chi\in S_\chi} \chi)(x)}$ is a trivial rank 1 sheaf).

Thus 
\begin{align*}
\det\Hcan(\{\chi_1,&\dots,\chi_n\},\{1,\dots,1\})|_{G_{\Q(\mu_N)}}(\Frob_\cP) \\
&=\left(\frac{q^{n(n-1)/2}\left(\prod_{i}(-g(\bar{\psi},1/\bar{\chi_i}))\right)^n}{\left[\prod_i(-g(\psi,\chi_i))^n\right] (-g(\bar{\psi},1))^{n^2}}\right)^{[\bfk:\F]} 
\end{align*}

On the other hand, it is easy to see by a similar argument that
\begin{align*}
\lambdacan_{G_{\Q(\mu_N)}}(\{\chi_i\},&\{1\})(\Frob_\cP) =\left(\frac{\chi_i((-1)^{n-1}) (-g(\bar{\psi},1/\bar{\chi_i}))}{(-g(\psi,\chi_i))(-g(\bar{\psi},1))}\right)^{[\bfk:\F]} \cL_{\chi_i}(\Frob_{\cP,2^N})
\end{align*}
so
\begin{align*}
\prod_i \lambdacan_{G_{\Q(\mu_N)}}(\{\chi_i\},\{1\})(\Frob_\cP) &=\left(\frac{\prod_i\chi_i((-1)^{n-1})\prod_i(-g(\bar{\psi},1/\bar{\chi_i}))}
{\left[\prod_i(-g(\psi,\chi_i))\right] (-g(\bar{\psi},1))^{n}}
\right)^{[\bfk:\F]}     (\bigotimes_i \cL_{\chi_i})(\Frob_{\cP,2^N})\\
&=\left(\frac{\prod_i(-g(\bar{\psi},1/\bar{\chi_i}))}
{\left[\prod_i(-g(\psi,\chi_i))\right] (-g(\bar{\psi},1))^{n}}
\right)^{[\bfk:\F]}.
\end{align*}
(since $\prod_i\chi_i=1$, and $(\bigotimes_i \cL_{\chi_i})(\Frob_{\cP,2^N}) =  \cL_{\prod_i\chi_i}(\Frob_{\cP,2^N})$, which is trivial.)
Whence, dividing,
\begin{align*}
\left(\frac{\det\Hcan(\{\chi_1,\dots,\chi_n\},\{1,\dots,1\})|_{G_{\Q(\mu_N)}}}{\prod_i (\lambdacan_{G_{\Q(\mu_N)}}(\{\chi_i\},\{1\}))^{n}}\right)(\Frob_\cP)&=q^{n(n-1)[\bfk:\F]} \\
&=\eps_l^{n(1-n)}(\Frob_\cP)
\end{align*}
This is the desired result.
\end{proof}

\subsection{The determinant}
\label{det as psi1 psi2}
We now turn to studying the determinant  $\det\Prim_l$, understanding of which will prove an important ingredient in the proof of our main theorem. Our main tool in doing so will be Corollary \ref{cor: relate prim and hcan} above, which relates the sheaves $\Prim_l$ to the `canonical' hypergeometric sheaves studied in the previous section.

We define a character $G_{\Q(\mu_N)}\ra\Q_l^\times$:
$$\phi_l:=\Lambda_{v}\prod_i (\lambdacan_{G_{\Q(\mu_N)}}(\{\chi_i\},\{1\}))^{2}$$
where the $\chi_1,\dots,\chi_n$ are the elements of the multiset $S'_\chi$ in the statement of Corollary \ref{cor: relate prim and hcan}, and where $\Lambda_v$ is the character there. 

Next, we adopt some notational conventions: the reader may wish to compare with \S\ref{sssec: some conventions and a lemma}. Using the rational point 2 on $T_0^{(M)}$, we write $\pi_1(T_0^{(M)})$ as $\pi_1(T_0^{(M)}\times\Q^{ac})\ltimes G_{\Q(\mu_N)}$ and observe that any character of $\pi_1(T_0^{(M)})$ factors through 
$$(\pi_1(T_0^{(M)}\times\Q^{ac})^{ab})_{G_{\Q(\mu_N)}}\times (G_{\Q(\mu_N)})^{ab}.$$
Thus we can write $\det\Prim_l$ as the product of two characters, $\det\Prim_l=\psi_1\psi_2$, where $\psi_1$ factors through $(\pi_1(T_0^{(M)}\times\Q^{ac})^{ab})_{G_{\Q(\mu_N)}}$ and $\psi_2$ through $G_{\Q(\mu_N)}$.
\begin{lemma} \label{det-lemma} We have that 
$$\psi_2 = \phi_l^{n} \eps_l^{n(1-n)/2}$$
\end{lemma}
\begin{proof}
We start from the displayed equation in Corollary \ref{cor: relate prim and hcan}:
$$j_{1,*} \Prim^{n-2}_l \cong j_{2,*} [N]^* \Hcan(S'_\chi,S'_\rho))\otimes \Lambda_v$$
Taking stalks at $t=2$, we see that
$$(j_{1,*} \Prim^{n-2}_l)_{t=2} \cong (j_{2,*} [N]^* \Hcan(S'_\chi,S'_\rho))\otimes \Lambda_v)_{t=2}.$$
So
$$(\Prim^{n-2}_l)_{t=2} \cong ([N]^* \Hcan(S'_\chi,S'_\rho))\otimes \Lambda_v)_{t=2}.$$
And therefore:
$$(\Prim^{n-2}_l)_{t=2} \cong (\Hcan(S'_\chi,S'_\rho))\otimes \Lambda_v)_{t=2^N}$$
where both sides are naturally Galois representations. Taking determinants, we get
$$(\det\Prim^{n-2}_l)_{t=2} \cong (\det\Hcan(S'_\chi,S'_\rho))\otimes \Lambda_v)_{t=2^N}.$$
Since we wrote $\pi_1(T_0^{(M)})$ as $\pi_1(T_0^{(M)}\times\Q^{ac})\ltimes G_{\Q(\mu_N)}$ using the rational point 2, the left hand side of the displayed equation is just $\psi_2$. On the other hand, the right hand side is (using the notation of \S\ref{sssec: some conventions and a lemma}) $\det\Hcan(S'_\chi,S'_\rho)|_{G_{\Q(\mu_N)}} \otimes \Lambda_v$, since in \S\ref{sssec: some conventions and a lemma} we used the rational point $2^N$ to consider $\pi_1(B)$ to be $G_{\Q(\mu_N)}\ltimes \pi_1(B\times\Q^{ac})$. Thus we see
\begin{align*}
\psi_2 &= \det\Hcan(S'_\chi,S'_\rho)|_{G_{\Q(\mu_N)}} \otimes \Lambda_v
\\
&=\epsilon_l^{n(1-n)/2}\prod_{\chi\in S'_\chi} \lambdacan_{G_{\Q(\mu_N)}} (\{\chi\},\{1\})^n 
\end{align*}
using Lemma \ref{detHcan lemma}. Comparing with the definition of $\phi_l$, we see this is as required.
\end{proof}
Looking at the Hodge-Tate number of either side of the equation above at a prime $\fl$ over $l$, and writing $\mathrm{HT}_\fl(\phi_l)$ for the Hodge-Tate number of $\phi_l$ at that place, we get
\begin{align*}
2\times(\vec{h}(\fl)+(\vec{h}(\fl)+1)+\dots+(\vec{h}(\fl)+n-1)) &= 2n\,\mathrm{HT}_\fl(\phi_l) + n(n-1) \\
(2\vec{h}(\fl)+n-1)n &= 2n\,\mathrm{HT}_\fl(\phi_l)+ n(n-1) \\
(2\vec{h}(\fl))n &= 2n\,\mathrm{HT}_\fl(\phi_l)
\end{align*}
and we deduce that $\mathrm{HT}_\fl(\phi_l)=\vec{h}(\fl)$. Thus we can use twisting by $\phi_l$ to shift the Hodge-Tate numbers of an arbitrary representation by $\vec{h}$. We will write, given an $l$-adic representation $r$,  $r(-\vec{h})$ for the twist of $r$ by this character $\phi_l$, and $r(\vec{h})$ for the twist by the inverse.

\subsection{A Galois descent}
We now need to prove a lemma which will play a small but critical role in the argument for the main theorem of this section. The reader may wish to skip these arguments at first reading, examine the proof of the main theorem at the end of the section, and having seen \emph{why} we need the result we are about to prove, return to read the proof of it.

The issue it resolves is as follows. We have said that the basic structure of the argument which allows us to find prescribed residual representations in the cohomology of the Dwork family is the following: we construct a moduli space of points in the family which admit such isomorphisms, then we show it has a point over a suitable field by applying the theorem of Moret-Bailly (the form of this theorem which we will use is Proposition 2.1 of \cite{hsbt}).  The trouble is that we want to ensure that the point we construct will exist over a CM-field. Whereas the theorem of Moret-Bailly lends itself well to constructing points over totally real fields (since this is expressible as a local condition), asking for a CM field is not possible. Thus we need a less direct approach.

The basic idea we will use is as follows. We will construct a scheme over a \emph{totally real} field $F^+$, which will parametrize isomorphisms which exist when one passes to a certain quadratic totally imaginary extension $F$ of that totally real field $F^+$. Moret-Bailly will allow us to show that this scheme has a point over a totally real extension field $F^{+,}{}'$---this will correspond to the isomorphism we need over a quadratic totally imaginary extension $F'$ of $F^{+,}{}'$, which will be what we want. 

Our goal is to prove a technical result which shows that a scheme parameterizing such isomorphisms does in fact exist.

\smallskip

Let us proceed to the actual setup. Suppose we have a base scheme $S_0^+$, defined over a totally real field $F^+$ which contains the totally real subfield $\Q(\mu_N)^+$ of $\Q(\mu_N)$. Let $F:=F^+(\mu_N)$ and let us write $S_0$ for the base change $S_0^+\times_{F^+}F$. Let $\chi$ be a character of $G_F$ into $(\Z/M\Z)^\times$. Suppose further that we have two lisse rank $n$ mod $M$ sheaves $\cA, \cB$ on $S_0$. Suppose also that $\cA, \cB$ satisfy $\cA^c\cong\cA^\vee\otimes\chi$, $\cB^c\cong\cB^\vee\otimes\chi$, where $\cA^c$ is the `complex conjugate' sheaf. (That is, the sheaf whose corresponding representation of $\pi_1(S_0)$ is $r\circ j_c$ where $r$ is the representation of $\pi_1(S_0)$ associated to $\cA$, and $j_c$ is the outer automorphism of $\pi_1(S_0)$ coming from conjugation by a complex conjugation of the totally real subfield.) 

Thinking of $\cA$ as a mod $M$ representation $V_\cA$ of $\pi_1(S_0)$, this is the same as giving a pairing $\langle*,*\rangle$ on $V_\cA$ which satisfies
$$\langle \sigma v_1,j_c(\sigma) v_2\rangle = \chi(\sigma) \langle v_1,v_2\rangle$$
and similarly for $\cB$. We will suppose in addition that these pairings are \emph{symmetric}. (That is, $\cA$ and $\cB$ have sign +1 in the sense of Bella\"iche-Chenevier: see \cite[\S1.1]{bc}.)

Suppose finally that there is an isomorphism $\eta: \wedge^n \cA \ra \wedge^n \cB$ and we have fixed one such isomorphism. This isomorphism should be compatible with the maps $\cA^c\cong\cA^\vee\otimes\chi$, $\cB^c\cong\cB^\vee\otimes\chi$ in the following sense. First note that $\cA^c\cong\cA^\vee\otimes\chi$ will induce a map $ (\wedge^n \cA)(\wedge^n \cA)^c\ra\chi^n$, and hence we get (using a similar map for $\cB$) an distinguished isomorphism $ (\wedge^n \cA)(\wedge^n \cA)^c\cong (\wedge^n \cB)(\wedge^n \cB)^c$ (since both have specified isomorphisms to $\chi^n$). $\eta$ will also induce an isomorphism $ (\wedge^n \cA)(\wedge^n \cA)^{c}\cong (\wedge^n \cB)(\wedge^n \cB)^{c}$; we ask that these agree. 

There is a certain important circumstance in which we can arrange for a compatible isomorphism $\eta$ to exist. Suppose that we have \emph{some} isomorphism $\eta': (\wedge^n \cA) \ra (\wedge^n \cB)$, and suppose moreover that there is some map $\nu: \cA\isoto\cB$ which is an isomorphism in the category of vector spaces equipped with a pairing but no additional structure. (That is, this isomorphism $\nu$ need not respect the Galois action at all, but does form a commutative square
$$\xymatrix{
\cA^c \ar[r]\ar[d]^{\nu} & \cA^\vee \otimes \chi\\
\cB^c \ar[r] & \cB^\vee \otimes \chi \ar[u]^{\nu^\vee}
}$$
with the maps coming from our chosen isomorphisms $\cA^c\cong\cA^\vee\otimes\chi$, $\cB^c\cong\cB^\vee\otimes\chi$.) Then taking the $\wedge^n$ of $\nu$ we can construct an isomorphism $\eta$ of 1-dimensional vector spaces (without Galois action) $(\wedge^n\cA)\ra(\wedge^n\cB)$. Now the key point: \emph{given that $\eta'$ exists, this $\eta$ will automatically respect the Galois action}. (This is because the existence of $\eta'$ tells us that the characters by which Galois acts on each side are identical, which will force any isomorphism between $(\wedge^n\cA)$ and $(\wedge^n\cB)$ in the category of vector spaces to also be an isomorphism in the category of vector spaces equipped with a Galois action.) It is also immediate, given the commutative diagram above in the construction of $\eta$, that it is compatible with the isomorphisms $\cA^c\cong\cA^\vee\otimes\chi$, $\cB^c\cong\cB^\vee\otimes\chi$ in the sense we require.

Now, given a scheme $R^+$ over $S_0^+$ we can base change to form a scheme $R:=R^+\times_{S_0^+} S_0$ over $S_0$. We can define a functor
\begin{align*}
S_{\cA,\cB}:\left\{S^{+}_0\text{-schemes}\right\} & \rightarrow \mathbf{Set} \\
R^+&\mapsto\left\{\parbox{6.8cm}{Isomorphisms $\xi$ between the pull back to $R$ of $\cA$ and $\cB$, such that $(\wedge^n\xi)=\eta$.} \right\} 
\end{align*}
(Here `isomorphisms' means isomorphisms of sheaves \emph{with stipulated pairings $\langle*,*\rangle$}.)

\begin{proposition} This functor is represented by a scheme. \label{representable-lemma}
\end{proposition}
\begin{proof} We will begin by constructing a certain finite \'etale cover $S_1$ of the scheme $S_0^+$; we will then show that this $S_1$ represents the functor we want.

We can specify an finite \'etale cover of $S_0$ by giving a representation of $\pi_1(S_0^+)$ into the symmetric group on $Q$ letters, where $Q$ is the number of sheets, or equivalently by giving an action of $\pi_1(S_0^+)$ on a $Q$ element set. We can think of $\cA$ and $\cB$ as giving mod $M$ representations of $\pi_1(S_0)$, say acting on the free $\Z/M\Z$ modules $V_\cA$ and $V_\cB$ respectively. Thus we can immediately construct an \'etale cover of $S_0$ by allowing $\pi_1(S_0)$ to act on the finite set $X$ of isomorphisms of vector spaces $\iota:V_\cA \mapsto V_\cB$, via the action $A$ given by
\begin{align} \label{act.eq}
\begin{split}
\pi_1(S_0) \times X &\ra X\\
(\alpha, \iota) &\mapsto \alpha^{-1}\iota\alpha 
\end{split}
\end{align}
(and indeed, it is easy to see that this corresponds to the variety $\mathrm{Isom}(\cA,\cB)$ over $S_0$ parameterizing isomorphisms between $\cA$ and $\cB$ ignoring the pairing $\langle*,*\rangle$). If we replaced the set $X$ with the smaller set $X_\eta$ of isomorphisms whose induced map on $\wedge^n$'s is $\eta$, then we would get the variety parameterizing isomorphisms lifting $\eta$.

We wish, however, to construct an \'etale covering of $S_0^+$, which means we need to extend the above action to an action of $\pi_1(S^+_0)$. Now, if we write $c$ for complex conjugation $c\in\pi_1(S^+_0)$, then $\pi_1(S^+_0)$ is generated by $c$ and $\pi_1(S_0)$; so we just need to define an action of $c$ on $X_\eta$ which commutes in the right way with all the other actions we have defined.

Given an isomorphism $\iota:V_\cA \mapsto V_\cB$, we can define an isomorphism $\tilde{\iota}$ as follows: for all $v_1,v_2 \in V_\cA$, we impose $\langle \iota v_1, \tilde{\iota} v_2\rangle=\langle v_1,v_2\rangle$. (Thus $\tilde{\iota}$ is the `inverse of the adjoint' of $\iota$.) We can easily calculate that $\tilde{\tilde{\iota}}=\iota$, since:
\begin{align*}
\langle \tilde{\iota} v_1, \tilde{\tilde{\iota}} v_2\rangle &= \langle v_1,  v_2\rangle = \mathrm{sgn}\,V_\cA\, \langle v_2,  v_1\rangle 
 = \mathrm{sgn}\,V_\cA\, \langle \iota v_2,  \tilde{\iota} v_1\rangle\\
&= \mathrm{sgn}\,V_\cA\,\mathrm{sgn}\,V_\cB\, \langle  \tilde{\iota} v_1, \iota v_2\rangle = \langle  \tilde{\iota} v_1, \iota v_2\rangle
\end{align*}
(\emph{note that at this point we use the fact that both $\cA$ and $\cB$ have sign +1; or, more precisely, that they have the same sign}). Moreover, we note that for $\alpha\in\pi_1(S_0)$, we have $(\alpha^{-1}\iota\alpha)\tilde{\,}= j_c(\alpha)^{-1}\tilde{\iota}j_c(\alpha)$ where $j_c(\alpha)$ as above denotes conjugation by complex conjugation; the demonstration goes as follows: 
\begin{align*}
\langle \alpha^{-1}\iota\alpha v_1, j_c(\alpha)^{-1}\tilde{\iota}j_c(\alpha) v_2 \rangle &= \chi(\alpha^{-1}) \langle \iota \alpha v_1, \tilde{\iota} j_c(\alpha) v_2 \rangle \\
&= \chi(\alpha^{-1}) \langle  \alpha v_1,  j_c(\alpha) v_2 \rangle \\
&= \chi(\alpha^{-1})\chi(\alpha) \langle  v_1,   v_2 \rangle = \langle  v_1,   v_2 \rangle 
\end{align*}
These two relations ensure that we can extend our action $A$ on $X_\eta$ to an action of $\pi_1(S^+_0)$ by stipulating that $A(c)(\iota)=\tilde{\iota}$. (The fact that this action preserves $X_\eta$ inside $X$ is a consequence of the fact that we chose the isomorphism $\eta$ compatibly with the pairings on $\cA,\cB$.) Hence we have constructed an \'etale cover of $S_0^+$, which we will call $S^+_\eta$.

We now pass to consider the question of what it means to give a $S^+_\eta$-scheme, say $f:R^+\ra S^+_\eta$. From general facts about \'etale covers, this is the same as giving a $S^+_0$-scheme $f_0:R^+\ra S^+_0$ together with a point in $X_\eta$ which is stabilized by the image of $\pi_1(R^+)$ in $\pi_1(S^+_0)$ under the map on $\pi_1$ induced by $f_0$. Now, given such a map $f_0$, pullback induces a map $R\ra S_0$ and we will have a commutative diagram:
$$\xymatrix{
\pi_1(R^+) \ar[rr]^{f_0} && \pi_1(S^+_0) \ar[r] & \mathrm{Sym}(X_\eta)\\
\pi_1(R) \ar@{^{(}->}[u]^{\text{index 2}}\ar[rr] && \pi_1(S_0) \ar[ur]\ar@{^{(}->}[u]^{\text{index 2}}\\
}$$
To give a point in $X_\eta$ stabilized by the image of $f_0$ is to give
\begin{enumerate}
\item A point in $X_\eta$ stabilized by the image of $\pi_1(R)$ in $\pi_1(S_0)$...
\item ...which is also fixed by $c\in\pi_1(S^+_0)$.
\end{enumerate}
Now, point 1 here is equivalent (by e.g.~the remarks immediately after equation \ref{act.eq}) to giving an isomorphism $\theta$ between the pullbacks of $\cA$ and $\cB$ from $S_0$ to $R$ ignoring the pairing $\langle*,*\rangle$. Then point 2 imposes additionally that $\tilde{\theta}=\theta$; unpacking this, it is seen to be equivalent to $\theta$ preserving the pairing $\langle*,*\rangle$. This is as required.
\end{proof}

One final remark should be made in this connection. What does it mean to give a point of the scheme (or equivalently the functor) just defined over a field $K$ which \emph{contains} $\Q(\mu_N)$? A fairly easy check shows that this is just the same as giving an isomorphism between the pullback to $K$ of $\cA$ and the pullback to $K$ of $\cB$, now \emph{disregarding} the pairing.

\subsection{Realizing residual representations}
We are now in a position to prove a result allowing us to realize residual Galois representations in the cohomology of the family $Y_t$.

\begin{proposition} \label{geom-prop} The family $Y_t$ and the piece of its cohomology corresponding to $\Prim_{l,t}$ have the following property:

Suppose $K/F$ is a Galois extension of CM fields, with totally real subfields $K^+, F^+$, $n$ is a positive integer, $l_1, l_2 \dots l_r$ are distinct primes which are unramified in $K$, and that we are given residual representations
$$\bar{\rho}_i:\Gal (\Fbar/F) \ra \GL_n(\F_{l_i}).$$
Suppose further that we are given $\fq_1, \fq_2, \dots, \fq_s$, distinct primes of $F$ above rational primes $q_1,\dots,q_s$ respectively, and $\cL$ a set of primes of $F$ not including the $\fq_j$ or any primes above the $l_i$. Suppose that each $q_j$ satisfies $q_j\notdiv N$. Finally, suppose that the following conditions are satisfied for each $i$:
\begin{enumerate}
\item \label{geom-prop-cdx-l-cn}
$l_i>C(n, N)$

\item \label{geom-prop-cdx-l-mod}
$l_i\equiv 1 $ mod $N$

\item \label{geom-prop-cdx-unram-at-cl}
$\bar{\rho}_i$ is unramified at each prime of $\cL$ and at the $l_k$ for $k\neq i$.

\item \label{geom-prop-cdx-inertia}
For each prime $\fw$ above $l_i$, we have that 
$$\bar{\rho}_i | _{I_\fw} \cong 1\oplus \epsilon_{l_i}^{-1}\oplus\dots\oplus  \epsilon_{l_i}^{1-n}$$

\item \label{geom-prop-cdx-dual}
We have that there exists a polarization $\bar{\rho}_i^c\cong \bar{\rho}_i^\vee \eps_{l_i}^{1-n}$; given this, we can associate to $\bar{\rho}_i$ a sign in the sense of Bella\"iche-Chenevier and we require that this sign is +1. We also require that the polarization can be chosen so that its determinant is the same as the determinant of the polarization $\Prim_{l_i,t}^c \cong \Prim_{l_i,t}^\vee \eps_{l_i}^{2-N}$. Finally, we require that $\det\bar{\rho_i}\cong \eps_{l_i}^{n(1-n)/2}$
\end{enumerate}

Then we can find a CM field $K'/F$, linearly disjoint from $K/F$, a finite-order character $\chi_i:\Gal(\Qbar/K') \ra \Q_{l_i}$ for each $i$, and a $t\in K'$ such that, 
\begin{enumerate}
\item \label{geom-prop-conclude-unram-and-split}
All primes of $F$ above the $\{l_i\}_{i=1,\dots,r}$ and all the $\cL$ are unramified in $K'$
\item \label{geom-prop-conclude-good-reduction}
For all $i$, $Y_t$ has good reduction at each prime above lying above $l_i$, and each prime above the primes of $\cL$.
\item \label{geom-prop-conclude-crys-ht}
For all $i$ and $\fw|l_i$, $\Prim_{\fw,t}(\vec{h})\otimes\chi_i$ is crystalline with H-T numbers $\{0,1,\dots,n-1\}$.
\item \label{geom-prop-conclude-steinberg}
For each $\fQ$ above some $\fq_j$, we have that $(\Prim_{l_i,t})^{ss}$ and $\chi_i$ are unramified at $\fQ$, with $(\Prim_{l_i,t}^{ss}(\vec{h})\otimes\chi)(\Frob_{\fQ})$ having eigenvalues $\{1, \#k(\fQ), \#k(\fQ)^2,\dots,\#k(\fQ)^{n-1}\}$.
\item \label{geom-prop-conclude-agree-mod}
$\Prim[l_i]_{t}(\vec{h})\otimes\bar{\chi}_i = \bar{\rho}_i$ for all $i$.
\end{enumerate}
\end{proposition}
\begin{proof}
Throughout this proof, we will set $M=\prod l_i$.

Since $\GL_n(\Z/M\Z)$ is just $\GL_n(\Z/l_1\Z)\times\dots\times\GL_n(\Z/l_r\Z)$, we can combine the $\rho_i$ into a single representation
$$\rho_{\Z/M\Z}:\Gal (\Fbar/F) \ra \GL_n(\Z/M\Z)$$
and similarly we can combine the $\phi_l$'s mod $l$ for different $l_i$ too, to get a mod $M$ character; we will write `$(\vec{h})$' for the twist by this character also. 

We note that, thinking of $\Prim[M]$ and $\bar{\rho}_{\Z/M\Z}$ as $\Z/M\Z$ modules with pairing and Galois action, they are certainly isomorphic once we disregard the Galois action and only keep the pairing. (Since $\F_l$ vector spaces with pairing are classified by the determinant of the pairing, and since $\bar{\rho}_i$ and $\Prim[l_i]$ have polarizations with the same determinant, by hypothesis 5, this is immediate.)

Next, we must study the determinant  $\det\Prim_l$, a representation of $\pi_1(T_0^{(M)})$.  Recall that we have, in \S \ref{det as psi1 psi2} above, written $\det\Prim_l$ as the product of two characters, $\det\Prim_l=\psi_1\psi_2$, where $\psi_1$ factors through $(\pi_1(T_0^{(M)}\times\Q^{ac})^{ab})_{G_{\Q(\mu_N)}}$ and $\psi_2$ through $G_{\Q(\mu_N)}$.

But $\psi_1$ maps into the image of geometric monodromy under the representation of $\pi_1(T_0^{(M)})$ determined by $\det\Prim_l$, which we know to be trivial, since geometric monodromy acts on $\Prim_l$ via matrices in $\SL$.  Thus $\psi_1$ is trivial. And $\psi_2$ was studied above in Lemma \ref{det-lemma}. We deduce that
$\det\Prim[M](\vec{h}) = \phi^{-n}\det\Prim[M] = \phi^{-n}\psi_1\psi_2 = \epsilon_l^{n(1-n)/2}$

On the other hand, by hypothesis, we have that $\det\bar{\rho}_{\Z/M\Z} = \eps_l^{n(1-n)/2}$. Thus $\det\Prim[M] = \det\bar{\rho}_{\Z/M\Z}$, and we may fix a choice of isomorphism $\eta:\det\Prim[M](\vec{h})\ra \det\bar{\rho}_{\Z/M\Z}$. (Indeed, we can choose that this isomorphism be compatible with the polarizations on $\Prim[M]$ and $\bar{\rho}_{\Z/M\Z}$, in the sense defined in the previous section. As was discussed there, to prove that this is possible it will suffice to give an isomorphism $\bar{\rho}_{\Z/M\Z}\ra\Prim[M]$ as vector spaces with pairing but without Galois action, as was done above.)

These preliminaries done, we are now on to the heart of the proof. The basic method is to consider the moduli space of tuples $(Y_t,\iota)$ where $Y_t$ is an element of the family $\fF$, and $\iota$ is an isomorphism between  $\rho_{\Z/M\Z}$ and the mod $M$ cohomology of $Y_t$ twisted by the character $\bar{\phi}_{\Z/M\Z}$. We shall show that this has a point over a large totally real field using the theorem of Moret-Bailly.

Let us proceed with the details. It will be useful to give a name to the totally real analogue of our base space $T^{(M)}_0$; so let us define $R_0^+$ to be $\Z[\mu_N,\frac{1}{N}]^+$ and $T^{(M)+}_0$ to be $\Spec R_0^+[\frac{1}{\lambda^N-1},\lambda,\frac{1}{M}]$. Now, let $\cW$ be a free $\Z/M\Z$-module of rank $n$ with a continuous action of $\Gal(\bar{F}/F)$; we can think of this as a lisse etale sheaf on $\Spec F$. In particular, we will be taking $\cW$ to be the module coming from $\bar{\rho}_{\Z/M\Z}(-\vec{h})$. Given a $T^{(M)+}_0\times_{\Z[\mu_N]^+}\Spec F^+$ scheme $S^+$, we can pull back along 
$$T^{(M)}_0\times_{\Z[\mu_N]}\Spec F \rightarrow T^{(M)+}_0\times_{\Z[\mu_N]^+}\Spec F^+$$
to get a $T^{(M)}_0\times_{\Z[\mu_N]}\Spec F$ scheme $S$, and we can consider isomorphisms between the pullback of $\cW$ to $S$ and the pullback of $\Prim[M]$ to $S$.

Consider the functor $T_\cW$:
\begin{align*}
\left\{T^{(M)+}_0\times_{\Z[\mu_N]^+}\Spec F^+\text{-schemes}\right\} & \rightarrow \mathbf{Set} \\
S^+&\mapsto\left\{\parbox{6.6cm}{Isomorphisms $\xi$ between the pull back to $S$ of $\cW$ and of $\Prim[M]$ such that the induced isomorphism $(\det\xi): (\det\Prim[M])\ra(\det \cW)$ agrees with $\eta$.} \right\} 
\end{align*}
This functor is represented by a scheme, which we will also denote by $T_\cW$. (To see this, we simply apply Proposition \ref{representable-lemma}.)

We then have the following facts:
\begin{enumerate}
\item \emph{The scheme $T_\cW$ is geometrically connected}. To see this, we must see that the geometric monodromy acts transitively on the points in a fiber of $T_\cW\ra T_0^{(M)}$. This fiber is the set of isomorphisms between the rank $n$ $\Z/M\Z$ modules $\Prim[M]$ and $\cW$ which preserve the determinant; any such isomorphism can be transformed into any other by the action of $\SL_n(\Z/M\Z)$. But we are then done by Corollary \ref{monodromy-mod-prop}.
\item \emph{If we let $S_1$ denote the set of infinite places, and define 
\begin{align*}
\Omega_{w}&=T_\cW^{(M)}(F^+_w)
\end{align*}
(where $w$ refers to an infinite place) then these sets are nonempty.} We claim that this has a point over $0\in T_0^{(M)}$. To give such a point is to give an isomorphism between the pullbacks of $\Prim[M]_0$ and $\cW$ to $\mathbb{R} \otimes_\mathbb{F^+} F$; that is, to $\C$. But once we pull back to $\C$, all Galois action information is discarded, and all that remains are spaces with a pairing---and we saw these to be isomorphic at the beginning of the proof. 
\item \emph{If we let $S_2$ denote the set of primes above the $l_i$ together with the primes of $\cL$, and define, for $\fw\in S_2$
\begin{align*}
\Omega_{\fw}&=\{  t^* \in T_0^{(M)}(F^{+nr}_\fw) \text{ above } t \in T_\cW^{(M)}(F^{+nr}_\fw) \text{ s.t. } v_{\fw}(1+t^{N}) < 0\} 
\end{align*}
then these sets are nonempty.}

To see that these sets are isomorphic, we will actually show that there is a point in the sets above lying above the point $0\in T_0^{(M)}$; that is, we will show that the Galois representations $\Prim[M]_0$ and $\cW$ become isomorphic once restricted to the absolute Galois group of $(F^+_\fw)^{nr}$; or, in other words, once restricted to inertia. To see this, first use condition (\ref{geom-prop-cdx-unram-at-cl}), which gives us what we require at $\cL$. (Both representations are unramified, so trivial on inertia) Then use condition (\ref{geom-prop-cdx-inertia}) at the places above the $l_i$, which tells us that the inertial representation of $\cW=\bar{\rho}_{\Z/M\Z}(-\vec{h})$ at a prime $\fw$ above $l_i$ is a direct sum of increasing powers of the cyclotomic character, starting with the $h(\sigma)$'th power, where $\sigma:F\ra\Qbar_l$ is the embedding corresponding to $\fw$; and condition (\ref{geom-prop-cdx-l-mod}) together with conclusion (\ref{geom-lemma-inertia}) of Proposition \ref{geom-lemma} which tells us that $\Prim_{l,0}$ takes exactly the same form.

\item \emph{If we let $S_3$ denote the set of the $\fq_j$, and define 
\begin{align*}
\Omega_{\fq_j}&=\{t^*\in T_\cW^{(M)}(\bar{F}_{\fq_j}) \text{ above } t \in T_0^{(M)}(\bar{F}_{\fq_j}) \text{ s.t. } v_{\fq_j}(t) < 0\} 
\end{align*}
then these sets are nonempty.}

Again, we will show that there is a point in the set lying above the point $0\in T_0^{(M)}$. This is immediate, since the mod $M$ representations $\cW$ and $\Prim[M]$ have finite image, and once we trivialize both by making a large local extension, they are isomorphic. 
\end{enumerate}

Thus, by the theorem of Moret-Bailly, in the version given as Proposition 2.1 of \cite{hsbt}, we can find a field $K'^+/F^+$, disjoint from $K^+/F^+$, and a point $t^*\in T_\cW(K')$ (where $K':=K'^+ F$) lying above a point $t$ in in $T_0^{(M)}(K')$ such that:
\begin{itemize}
\item All primes of $S_2$ (that is, all the primes above the primes $l_i$ and the primes of $\cL$) are unramified in $K'$. \emph{Thus we get conclusion (1).}
\item All primes of $S_1$ split completely in $K'$.  \emph{Thus we conclude that $K'^+$ is totally real and hence $K'$ is CM.}
\item For each $j$, we have $t\in \Omega_{\fq_j}$; that is, for each $j$ and for each prime $\fQ$ above $\fq_j$, we have that $v_{\fQ}(t) < 0$. Thus, by part (\ref{geom-lemma-steinberg}) of Proposition \ref{geom-lemma}, we can conclude for each $i$ that $(\Prim_{l_i,t})^{ss}$ is unramified at $\fQ$ and $(\Prim_{l_i,t})^{ss}$ has $\Frob_{\fQ}$ eigenvalues $\{\beta_{i,\fQ},\beta_{i,\fQ}\#k(\fQ),\beta_{i,\fQ}(\#k(\fQ))^2,\dots,$ $\beta_{i,\fQ}(\#k(\fQ))^{n-1}\}$ for some $\beta_{i,\fQ}$. Making a further totally-real field extension unramified at the $l_i$, we can assume that, for each $i$, all the $\beta_{i,\fQ}$ are 1 mod $l_i$.

We can then choose a character $\chi_i:\Gal(\Qbar/K')\ra\Qbar_{l_i}$ for each $i$ lifting $\bar{\chi_i}$ which is unramified at the primes of $\cL$, the primes above the $l_i$, and the $\fQ$ and which takes $\Frob_{\fQ}$ to $\beta_{i,j}^{-1}$.

Then it is immediate that $(\Prim_{l_i,t}^{ss}(\vec{h})\otimes\chi_i)(\Frob_{\fQ})$ has eigenvalues $\{1,\#k(\fQ),(\#k(\fQ))^2,\dots,(\#k(\fQ))^{n-1}\}$. \emph{Thus we get conclusion (4).}
\item We have, for each prime $\fw$ above either some $l_i$ or some element of $\cL$, that $t\in \Omega_{\fw}$; that is, $\fw(1-t^{N}) < 0$. Thus, by part \ref{geom-lemma-good-red} of Proposition \ref{geom-lemma}, $Y_t$ has good reduction at $\fw$ and $\Prim_{\fw,t}$ is crystalline. The Hodge-Tate numbers are $\{h(\sigma),h(\sigma)+1,\dots,h(\sigma)+n-1\}$ by part \ref{geom-lemma-ht} of Proposition \ref{geom-lemma}, where $\sigma:F\ra\Qbar_l$ is the embedding corresponding to $\fw$. Thus $\Prim_{\fw,t}(\vec{h})\otimes\chi_i$ is crystalline with Hodge-Tate numbers $\{0,\dots,n-1\}$. (Recall $\chi_i$ is finite order and unramified at the the $l_i$.) \emph{This gives us conclusions (\ref{geom-prop-conclude-good-reduction}) and (\ref{geom-prop-conclude-crys-ht})  of the present proposition.}
\end{itemize}
Finally, by definition of $T_W$, the point $t^*$ gives us a specified isomorphism between $\chi^{-1}_{\Z/M\Z} \otimes \bar{\rho}_{\Z/M\Z}(-\vec{h})$ and $\Prim[M]_t$; that is, we have
$$\Prim[M]_t(\vec{h})\otimes \bar{\chi}_{\Z/M\Z} =\bar{\rho}_{\Z/M\Z}$$
which is the final conclusion (\ref{geom-prop-conclude-agree-mod}) of the present proposition. This concludes the proof.
\end{proof}

We close this section with a short argument showing that the natural polarization on $\Prim[l]_{0}$ coming from Poincare duality will have determinant a square for the $l$ splitting in a certain quadratic extension of $\Q(\mu_N)$
\begin{proposition} \label{split-det-pol-prop}Suppose $N$ is a positive integer; then there is a quadratic extension $F^*(n,N)$ of $\Q(\mu_N)$ such that for any $l$ splitting in $F^*(n,N)$, the natural polarization on $\Prim[l]_{0}$ has determinant a square.
\end{proposition}
\begin{proof}
Choose an arbitrary infinite place of $\Q(\mu_N)$, and consider $H_\text{sing}(Y_0\times \C, \Z)$, the singular cohomology of the Fermat hypersurface $Y_0$ with integral coefficients. We can extend coefficients to $\mathcal{O}_{\Q(\mu_N)}$, getting $H_\text{sing}(Y_0\times \C, \mathcal{O}_{\Q(\mu_N)})$, which will break up into eigenspaces under the action of the group $\Gamma_W/\Delta$. Let $H_\text{sing}(Y_0\times \C, \mathcal{O}_{\Q(\mu_N)})_v$ denote the eigenspace corresponding to $v$. This will have a perfect integral Poincare duality pairing with $H_\text{sing}(Y_0\times \C, \mathcal{O}_{\Q(\mu_N)})_{(-v)}$, which is the complex conjugate of  $H_\text{sing}(Y_0\times \C, \mathcal{O}_{\Q(\mu_N)})_v$; combining Poincare duality with complex conjugation, we get a perfect integral pairing on $H_\text{sing}(Y_0\times \C, \mathcal{O}_{\Q(\mu_N)})_v$ itself, which will have a determinant, a well-defined element $\alpha$ of $\mathcal{O}_{\Q(\mu_N)}$. Let $F^*(n,N)=\Q(\mu_N,\sqrt{\alpha})$.

Now, the determinant of the Poincare duality pairing on $\Prim[l]_{0}$ is the same as the determinant of the pairing on $H_\text{\'et}(Y_0\times\C, \Z_l)$, (passing to the infinite place we chose discards the Galois action but leaves the pairing unaffected). This is, by the comparison theorem, the same as the determinant of the pairing on $H_\text{sing}(Y_0\times\C, \Z_l)$, which will be $\alpha$, considered as an element of $\Z_l$. (Recall $\alpha$ was the determinant of the pairing on $H_\text{sing}(Y_0\times\C, \mathcal{O}_{\Q(\mu_N)})$.) If $l$ splits in $F^*(n,N)$, then $\alpha$ mod $l$ is a square in $\F_l$, and hence we are done.
\end{proof}

\section{Constructing a `seed' Galois representation}\label{sec:seed}

\subsection{} In our proof strategy above, we had as step 4 the establishment of a good supply of mod $l'$ representations $\bar{r}'$ which have the powerful property that an $l'$-adic Galois representation which satisfies certain regularity properties and agrees with $\bar{r}'$ will automatically be automorphic. Our goal in this section is to state and prove a precise version of this fact.
 
\begin{proposition} \label{seed-prop}
Suppose that $F$ is a CM field, $n$ and $N$ are positive even integers, $l$ is a prime which is unramified in $F$, and that we are given a representation
$$r:\Gal (\Fbar/F) \ra \SL_n(\Z_l)$$
Suppose further that $v_q$ is a prime of $F$ above a rational prime $q\neq l$ and $\cL$ be a finite set of primes of $F$ not containing primes above $lq$. Then we can find a rational prime $l'$ and a mod $l'$ representation 
$$\bar{r}':\Gal (\Fbar/F) \ra \GSp_n(\F_{l'})$$
with multiplier $\eps_l^{1-n}$, which satisfy the following conditions:
\begin{enumerate}

\item \label{seed-prop-conclude-lprime}
$l'>C(n,N)$, $l' \equiv 1$ mod $4N$, and $l'$ splits in $F^*(n,N)$. (Recall that the constant $C(n,N)$ was defined in Corollary \ref{monodromy-mod-prop}.) 
\item \label{seed-prop-conclude-rprime-unram}
$\bar{r}'$ unramified at all primes of $\cL$ and above $l$.
\item \label{seed-prop-conclude-inertia}
For each prime $\fw$ of $F$ above $l'$, we have that 
$$\bar{r}'|_{\Gal(\Fbar_{\fw}/F_0)} \cong 1\oplus \epsilon_{l'}^{-1}\oplus\dots\oplus  \epsilon_{l'}^{1-n}$$

\item \label{seed-prop-conclude-r-at-lprime} 
$\bar{r}$ unramified at $l'$.
\item \label{seed-prop-conclude-some-things-modular}
Whenever $F'/F$ is a field extension and $r'':\Gal (\Fbar/F') \ra \GL_n(\Z_{l'})$ is  a $l'$-adic Galois representation which satisfies the following conditions:
\begin{enumerate}
\item We have that $r'' \cong (r'|_{\Gal (\Fbar/F')})$ mod $l'$.
\item $r''^c \cong r''^\vee \eps_l^{1-n}$
\item $r''$ ramifies at only finitely many primes
\item For all places $v|l$ of $F$, $r''|_{\Gal(\Fbar_v/F_v)}$ is crystalline.
\item For all $\tau\in{\Hom(F,\Qbar_l)}$ above a primes $v|l$ of $F$,
\[
\dim_{\Qbar_l} \gr^i(r'' \otimes_{\tau,F_v} \BDR)^{\Gal(\Fbar_v/F_v)} = 
\begin{cases}
     0&  (i=0,1,\dots,n-1)\\ 
     1&  (\text{otherwise})
  \end{cases}
\]
\item For some prime $\fQ$ above $v_q$, we have that $r''|_{\Gal(\Fbar_{\fQ}/F_{\fQ})}$ is unramified, with $r''|_{\Gal(\Fbar_{\fQ}/F_{\fQ})}(\Frob_{\fQ})$ eigenvalues $\{\#k(\fQ)^j: j=0,\dots,n-1\}$ (for some $\alpha\in\Qbar_l^\times$).
\end{enumerate}
then $r''$ is automorphic over $F$ of weight 0 and type $\{\Sp_n(1)\}_{\{\fQ\}}$.
\end{enumerate}
\end{proposition}

\begin{proof}[Proof of Proposition \ref{seed-prop}]
As mentioned above, we are lucky in that the argument we need is entirely contained in the earlier work \cite{hsbt} and \cite{t-mod}. The facts we need from \cite{t-mod} are in a readily-citable form, but the arguments we need from \cite{hsbt} are not, being part of a longer argument (roughly speaking, they are the first three pages, pp 22--25, in the proof of Theorem 3.1). We will therefore briefly describe exactly what we need to take from \cite{hsbt} and then go on to cite the results we need from the other paper.

We begin following the argument at the beginning of Theorem 3.1 of \cite{hsbt}, taking $r=1$, $n_1=n$, (indeed, from now on we will often without further comment write $X$ where \cite{hsbt} writes $X_1$, for symbols $X$), and $F_0=F$ (all other notation being the same). Choose $E,M,\phi, l',\tilde{M},\tilde{w}_{l'}, w_{l'}$ as in \cite{hsbt} (except that when we choose $l'$, we make sure that it splits in $F^*(n,N)$, as we trivially may). Construct $\psi_{l'}$ as given by the recipe in the displayed equation on page 24, and use this to construct the character $\bar{\theta}$ with the properties in the middle of page 24. Finally, construct $I(\bar{\theta})$.

We have now taken all we require from \cite{hsbt}. $I(\bar{\theta})$ is the representation $r'$ we are seeking. (It has multiplier $\eps_{l'}^{1-n}$ from the first bullet point on page 24.) 
Point 1 comes from the first two bullet points in the second set of bullet points on page 23 (and the fact that $l'$ splits in a field containing $\zeta_{N}$); and point 4 comes from the fourth bullet there.
Points 2 and 3 comes from the first three bullet points concerning $\bar{\theta}$ on page 24.

Now we will prove part 5; this is where we appeal to \cite{t-mod}. Suppose that we are given such a representation $r''$. We will show $r''$ automorphic by appeal to Theorem 5.6 of \cite{t-mod} Conditions (1), (2), (3), (4), and (5) of that theorem are met by points (a-e) respectively. Condition (6) is immediate from point (f).
\end{proof}

\section{Putting the pieces together}\label{sec:final}

\subsection{} We are now in a position to use the various pieces we have accumulated to prove the main Theorem \ref{orig-main-theorem}. We will begin by reminding ourselves of the precise statement of the theorem. In the statement at the beginning of this paper, I  tried to group the conditions in a way that will be of maximum use to users of the theorem. But as we proceed to prove the theorem it will be useful to group the conditions in a different fashion, that reflects how they will be used in the proof. We will therefore provide a restated version of the theorem with the conditions regrouped to this end. The reader should have little  difficulty in convincing themselves that the two theorems are the same. 

\begin{theorem}[Restatement of Theorem \ref{orig-main-theorem}]\label{main-theorem}
Suppose that $F$ is a Galois extension of CM fields, $n$ is a positive even integer, $N\geq n+6$ is a positive even integer such that $F$ contains $\mu_{N}$, $l$ is a prime which is unramified in $F$, and that we are given a representation
$$r:\Gal (\Fbar/F) \ra \GL_n(\Z_l)$$
Suppose further that $v_q$ is a prime of $F$ above a rational prime $q\neq l$ and $\cL$ be a finite set of primes of $F$ not containing primes above $lq$, and that the following conditions are satisfied:
\begin{description}
\item[A] \label{cdx-a}
$(r|_{\Gal(\Fbar_{v_q}/F_{v_q})})\ssm $ is unramified and  $(r|_{\Gal(\Fbar_{v_q}/F_{v_q})}) \ssm$ has Frobenius eigenvalues $1, (\#k(v_q)), \dots, (\#k(v_q))^{n-1}$

\item[B1] $r^c \cong r^\vee\eps_l^{1-n}$, with sign +1, and with some choice of polarization having determinant a square
\item[B2] $r$ ramifies only at finitely many primes.
\item[B3] For each prime $w|l$ of $F$, $r|_{\Gal(\Fbar_w/F_w)}$ is crystalline with Hodge-Tate numbers $\{0,1,\dots, n-1\}$. 
\item[B4] $\Fbar^{\ker\ad \bar{r}}$ does not contain $F(\zeta_l)$
\item[B5] Let $\bar{r}$ denote the reduction of $r$; then $\bar{r}(\Gal(\Fbar/F(\zeta_l))$ is `big' in the sense of `big image'. 
\end{description}
\begin{description}
\item[C1] We have that $q\notdiv N$
\item[C2] $l>C(n,N)$. (This constant was defined in Corollary \ref{monodromy-mod-prop}.)
\item[C3] $l\equiv 1 \mod N$, and $l$ splits in the extension $F^*(N,n)$ 
\item[C4] $r$ is unramified at all the primes of $\cL$
\item[C5] We have that: 
$$\bar{r}|_{I_{F_w}} \cong 1 \oplus \eps_l^{-1} \oplus \dots \oplus  \eps_l^{1-n} $$
\item[C6] We have $(\det \bar{r})^2\cong \eps_l^{n(1-n)}$ mod $l$

\end{description}

Then there is a CM field $F'$ containing $F$ and linearly independent from  $\Fbar^{\ker \bar{r}}$ over $F$. In addition, all primes of $\cL$ and all primes of $F$ above $l$ are unramified in $F'$. Finally, there is a prime $w_q$ of $F'$ over $v_q$ such that $r|_\Gal{\Fbar/F'}$ is automorphic of weight 0 and type $\{\Sp_n(1)\}_{\{w_q\}}$.

Moreover, if at the same time we are given $F$ we are given a CM subfield $F_0$ of $F$ which also contains $\mu_N$, then we can additionally arrange that $F'$ is Galois over $F_0$.
\end{theorem}

I will also reproduce the lifting theorem which I need to apply from \cite{cht}: we have to refer constantly to the conditions of this theorem, and so it is convenient to have a statement of the theorem to hand.
\begin{theorem}[Theorem 5.2 of \cite{t-mod}] \label{cht-lifting-thm}
Let $F$ be an imaginary CM field and let $F^+$ be its maximal totally real subfield. Let $n\in \Z_{\geq 1}$ and let $l>n$ be a prime which is unramified in $F$. Let
$$r: \Gal(\Fbar/F)\ra GL_n(\Qbar_l)$$
be a continuous irreducible representation with the following properties. Let $\bar{r}$ denote the semisimplification of the reduction of $r$. Suppose that:
\begin{enumerate}
\item $r^c \cong r^\vee \eps_l^{1-n}$
\item $r$ is unramified at all but  finitely many primes.
\item For all places $v|l$ of $F$, $r|_{\Gal(\Fbar_v/F_v)}$ is crystalline.
\item There is an element $a\in (\Z^n)^{\Hom(F,\Qbar_l)}$ such that
\begin{itemize}
\item for all $\tau\in\Hom(F,\Qbar_l)$ we have either
$l-1-n\geq a_{\tau,1}\geq\dots\geq a_{\tau,n}\geq 0$
or
$l-1-n\geq a_{\tau c,1}\geq\dots\geq a_{\tau c,n}\geq 0$
\item for all $\tau\in\Hom(F,\Qbar_l)$ and all $i=1,\dots,n$ we have $a_{\tau c,i} = -a_{\tau,n+1-i}$
\item for all $\tau\in\Hom(F,\Qbar_l)$ above a prime $v|l$ of $F$
\[
\dim_{\Qbar_l} \gr^i(r \otimes_{\tau,F_v} \BDR)^{\Gal(\Fbar_v/F_v)} = 
\begin{cases}
     0&  i=a_{\tau,j}+n-j\,\text{(for some $j$)}\\ 
     1&  \text{(otherwise)}
  \end{cases}
\]
\end{itemize}
\item Let $r_l$ denote the local Langlands correspondance, normalized as in Proposition 4.3.1 of \cite{cht}, and $|\,\,|$ denote the modulus character. There is a non-empty finite set $S$ of places of $F$ not dividing $l$ and for each $v\in S$ a square integrable representation $\rho_v$ of $GL_n(F_v)$ over $\Qbar_l$ such that
$$(r|_{\Gal(\Fbar_v/F_v)})\ssm = r_l(\rho_v)^\vee(1-n)\ssm$$
If $\rho_v = \Sp_{m_v}(\rho'_v)$ then set
$$\tilde{r}_v=r_l((\rho'_v)^\vee |\,|^{(n/m_v - 1)(1-m_v)/2})$$
Note that $r|_{Gal(\Fbar_v/F_v)}$ has a unique filtration $\Fil^j_v$ such that 
$$\gr^j_v r|_{\Gal(\Fbar_v/F_v)} \cong \tilde{r}_v \eps^j$$ 
for $j = 0, \dots, m_v-1$ and equals $(0)$ otherwise. We assume that $\tilde{r}_v$ has irreducible reduction $\bar{r}_v$ . Then $\bar{r}|_{\Gal(\Fbar_v/F_v)}$ inherits a filtration $\overline{\Fil}^j_v$ with
$$\gr^j_v \bar{r}|_{\Gal(\Fbar_v/F_v)} \cong \bar{r}_v \eps^j$$ 
for $j = 0, ..., m_v-1$. 
\item $\Fbar^{\ker\ad \bar{r}}$ does not contain $F(\zeta_l)$
\item Let $r'$ denote the extension of $r$ to a continuous homomorphism $\Gal(\Fbar/F)\ra \cG_n(\Qbar_l)$, where $\cG_n$ is the group defined at the beginning of \cite{cht}; then $\bar{r}'(\Gal(\Fbar/F(\zeta_l))$ is `big'.\footnote{In the original statement of this theorem, the condition given is that `$\ad\bar{r}'(\Gal(\Fbar/F(\zeta_l))$ is big'. While the notion of `big image' is defined for a representation $r$, it is basically a property of the adjoint representation. Thus people often refer to $\ad \bar{r}$ as being big when they mean $\bar{r}$ is big. I will try to consistently use the $\bar{r}$ notation in this paper however.}
\item The representation $\bar{r}$ is irreducible and automorphic of weight $a$ and type $\{\rho_v\}_{v \in S}$ with $S \neq \emptyset$
\end{enumerate}

\end{theorem}

\begin{proof}[Proof of Theorem \ref{main-theorem}]
This is now a simple matter of combining the results we have accumulated according to our original strategy. (Note that the numbering of the steps here does not correspond directly to the numbering in the strategy.) Figure \ref{big-fig} may be of some help in understanding how the parts of the proof fit together. 
\begin{figure}
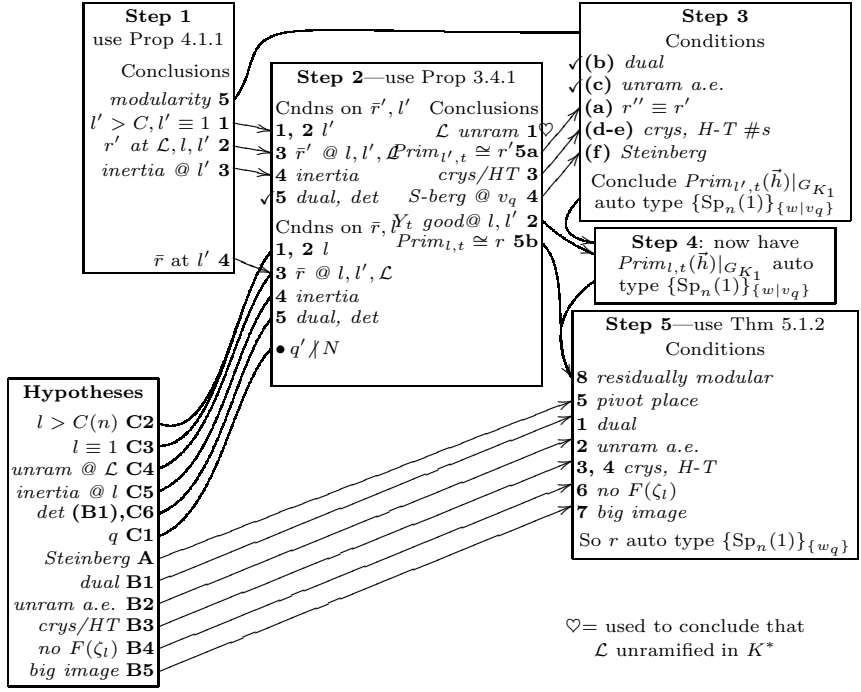

\def\objectstyle{\scriptstyle}
\def\labelstyle{\scriptstyle} 
 \xy
 (90,3)*{\text{$\heartsuit$= used to conclude that}};
 (90,0)*{\text{$\cL$ unramified in $K^*$}};
(75,45)*{}="S6tl"; (113,45)*{}="S6tr"; (75,12)*{}="S6bl";  (113,12)*{}="S6br"; 
(94,43)*{\text{{\bf Step 5}---use Thm \ref{cht-lifting-thm}}};
(94,40)*{\text{{Conditions}}};
(90,36)*{\text{\makebox[2.9cm][l]{\bf 8 \it residually modular}}};
(90,33)*{\text{\makebox[2.9cm][l]{\bf 5 \it pivot place}}};
(90,30)*{\text{\makebox[2.9cm][l]{\bf 1 \it dual}}};
(90,27)*{\text{\makebox[2.9cm][l]{\bf 2 \it unram a.e.}}};
(90,24)*{\text{\makebox[2.9cm][l]{\bf 3, 4 \it crys, H-T}}};
(90,21)*{\text{\makebox[2.9cm][l]{\bf 6 \it no $F(\zeta_l)$}}};
(90,18)*{\text{\makebox[2.9cm][l]{\bf 7 \it big image}}};
(94,14)*{\text{So $r$ auto type $\{\Sp_n(1)\}_{\{w_q\}}$}};
{\ar@/_.7pc/ (78,49)*{}; (75,36)*{}}; 
(71,54)*{}= "S2conc5b";(75,36)*{}= "S6cdx8";
"S2conc5b"; "S6cdx8" **\crv{(75,50)&(72,40)}; 
(71,24)*{}= "S3conc5";
(20,-3)*{} = "Hb5";(20,-6)*{} = "Hb6";
(112,21)*{}= "S6cdx6";(112,18)*{}= "S6cdx7";
{\ar (20,12)*{};(75,33)*{}};
{\ar (20,9)*{};(75,31)*{}};
{\ar (20,6)*{};(75,28)*{}};
{\ar (20,3)*{};(75,25)*{}};
{\ar (20,0)*{};(75,22)*{}};
{\ar (20,-3)*{};(75,19)*{}};
(78,56)*{}="S5tl"; (110,56)*{}="S5tr"; (78,46)*{}="S5bl";  (110,46)*{}="S5br"; 
(94,54)*{\text{{\bf Step 4}: now have}};
(94,51)*{\text{$\Prim_{l,t}(\vec{h})|_{G_{K_1}}$ auto}};
(94,48)*{\text{type $\{\Sp_n(1)\}_{\{w|v_q\}}$}};
{\ar@/_.7pc/ (76,60)*{}; (78,54)*{}}; 
{\ar@/_.1pc/ (71,57)*{}; (78,52.5)*{}}; 
(76,86)*{}="S4tl"; (112,86)*{}="S4tr"; (76,57)*{}="S4bl";  (112,57)*{}="S4br"; 
(94,84)*{\text{{\bf Step 3}}};
(94,81)*{\text{{Conditions}}};
(94,78)*{\text{\makebox[3.5cm][l]{\bf (b) \it dual}}};
(75.5,78)*{\checkmark};
(94,75)*{\text{\makebox[3.5cm][l]{\bf (c) \it unram a.e.}}};
(75.5,75)*{\checkmark};
(94,72)*{\text{\makebox[3.5cm][l]{\bf (a) \it $r''\equiv r'$}}};
(94,69)*{\text{\makebox[3.5cm][l]{\bf (d-e) \it crys, H-T $\#$s}}};
(94,66)*{\text{\makebox[3.5cm][l]{\bf (f) \it Steinberg}}};
(94,62)*{\text{Conclude $\Prim_{l',t}(\vec{h})|_{G_{K_1}}$}};
(94,59)*{\text{auto type $\{\Sp_n(1)\}_{\{w|v_q\}}$}};
(30,73)*{} = "S1cl10"; (76,82)*{}= "S4maincond";
"S1cl10"; "S4maincond" **\crv{(33,77)&(30,84)&(50,82)}; 
{\ar (71,66)*{}; (76,72)*{}}; 
{\ar (71,63)*{}; (76,69)*{}}; 
{\ar (71,60)*{}; (76,66)*{}}; 
(35,35)*{}="S2tl"; (71,35)*{}="S2tr"; (35,78)*{}="S2bl";  (71,78)*{}="S2br"; 
(53,76)*{\text{{\bf Step 2}---use Prop \ref{geom-prop}}};
(53,72)*{\text{\makebox[3.5cm][r]{Conclusions}}};
(53,69)*{\text{\makebox[3.5cm][r]{\it $\cL$ unram \bf \ref{geom-prop-conclude-unram-and-split}}}};
(71.5,69)*{\heartsuit};
(53,66)*{\text{\makebox[3.5cm][r]{\it $\Prim_{l',t} \cong r'$\bf \ref{geom-prop-conclude-agree-mod}a}}};
(53,63)*{\text{\makebox[3.5cm][r]{\it crys/HT \bf \ref{geom-prop-conclude-crys-ht}}}};
(53,60)*{\text{\makebox[3.5cm][r]{\it S-berg @ $v_q$ \bf \ref{geom-prop-conclude-steinberg}}}};
(53,57)*{\text{\makebox[3.5cm][r]{\it $Y_t$ good@ $l,l'$ \bf \ref{geom-prop-conclude-good-reduction}}}};
(53,54)*{\text{\makebox[3.5cm][r]{\it $\Prim_{l,t} \cong r$ \bf \ref{geom-prop-conclude-agree-mod}b}}};
(53,72)*{\text{\makebox[3.5cm][l]{Cndns on $\bar{r}', l'$}}};
(53,69)*{\text{\makebox[3.5cm][l]{\bf \ref{geom-prop-cdx-l-cn}, \ref{geom-prop-cdx-l-mod} \it $l'$}}};
(53,66)*{\text{\makebox[3.5cm][l]{\bf \ref{geom-prop-cdx-unram-at-cl} \it $\bar{r}'$ @ $l,l',\cL$}}};
(53,63)*{\text{\makebox[3.5cm][l]{\bf \ref{geom-prop-cdx-inertia} \it inertia}}};
(53,60)*{\text{\makebox[3.5cm][l]{\bf \ref{geom-prop-cdx-dual} \it dual, det}}};
(34.5,60)*{\checkmark};
(53,56)*{\text{\makebox[3.5cm][l]{Cndns on $\bar{r},l$}}};
(53,53)*{\text{\makebox[3.5cm][l]{\bf \ref{geom-prop-cdx-l-cn}, \ref{geom-prop-cdx-l-mod} \it l}}};
(53,50)*{\text{\makebox[3.5cm][l]{\bf \ref{geom-prop-cdx-unram-at-cl}  \it $\bar{r}$ @ $l,l',\cL$}}};
(53,47)*{\text{\makebox[3.5cm][l]{\bf \ref{geom-prop-cdx-inertia} \it inertia}}};
(53,44)*{\text{\makebox[3.5cm][l]{\bf \ref{geom-prop-cdx-dual} \it dual, det}}};
(53,40)*{\text{\makebox[3.5cm][l]{\bf  $\bullet\, q'\notdiv N$}}};
{\ar (30,70)*{}; (35,69)*{}}; 
{\ar (30,67)*{}; (35,66)*{}}; 
{\ar (30,64)*{}; (35,63)*{}}; 
{\ar (30,52)*{}; (35,50)*{}}; 
(20,30)*{} = "Hc1";(20,27)*{} = "Hc2";(20,24)*{} = "Hc4";(20,21)*{} = "Hc5";(20,18)*{} = "Hb1y";
(20,15)*{} = "Hcc";
(35,53)*{} = "S3cx12";(35,50)*{} = "S3cx3";(35,47)*{} = "S3cx4";(35,44)*{} = "S3cx6";
(35,40)*{} = "S3cq";
"Hc1"; "S3cx12" **\crv{(28,27)&(30,47)}; 
"Hc2"; "S3cx12" **\crv{(28,27)&(30,47)}; 
"Hc4"; "S3cx3" **\crv{(28,27)&(30,44)}; 
"Hc5"; "S3cx4" **\crv{(28,24)&(30,41)}; 
"Hb1y"; "S3cx6" **\crv{(28,21)&(30,38)}; 
"Hcc"; "S3cq" **\crv{(28,18)&(30,35)}; 
(10,50)*{}="S1tl"; (30,50)*{}="S1tr"; (10,86)*{}="S1bl";  (30,86)*{}="S1br"; 
(20,84)*{\text{\bf Step 1}};
(20,81)*{\text{use Prop \ref{seed-prop}}};
(20,77)*{\text{\makebox[1.9cm][r]{Conclusions}}};
(20,73)*{\text{\makebox[1.9cm][r]{{\it modularity} {\bf  \ref{seed-prop-conclude-some-things-modular}}}}};
(20,70)*{\text{\makebox[1.9cm][r]{{\it $l'>C,l'\equiv 1$} {\bf  \ref{seed-prop-conclude-lprime}}}}};
(20,67)*{\text{\makebox[1.9cm][r]{{\it $r'$ at $\cL,l,l'$} {\bf  \ref{seed-prop-conclude-rprime-unram}}}}};
(20,64)*{\text{\makebox[1.9cm][r]{{\it inertia @ $l'$} {\bf  \ref{seed-prop-conclude-inertia}}}}};
(20,52)*{\text{\makebox[1.9cm][r]{{$\bar{r}$ at $l'$} {\bf  \ref{seed-prop-conclude-r-at-lprime}}}}};
(0,-5)*{}="Htl"; (20,-5)*{}="Htr"; (0,36)*{}="Hbl";  (20,36)*{}="Hbr"; 
(10,34)*{\text{\bf Hypotheses}};
(10,30)*{\text{\makebox[1.9cm][r]{{$l>C(n)$} {\bf C2}}}};
(10,27)*{\text{\makebox[1.9cm][r]{{$l\equiv 1$} {\bf C3}}}};
(10,24)*{\text{\makebox[1.9cm][r]{{\it unram @ $\cL$} {\bf C4}}}};
(10,21)*{\text{\makebox[1.9cm][r]{{\it inertia @ $l$} {\bf C5}}}};
(10,18)*{\text{\makebox[1.9cm][r]{{\it det} {\bf (B1),C6}}}};
(10,15)*{\text{\makebox[1.9cm][r]{{\it $q$} {\bf C1}}}};
(10,12)*{\text{\makebox[1.9cm][r]{{\it Steinberg} {\bf A}}}};
(10,9)*{\text{\makebox[1.9cm][r]{{\it dual} {\bf B1}}}};
(10,6)*{\text{\makebox[1.9cm][r]{{\it unram a.e.} {\bf B2}}}};
(10,3)*{\text{\makebox[1.9cm][r]{{\it crys/HT} {\bf B3}}}};
(10,0)*{\text{\makebox[1.9cm][r]{{\it no $F(\zeta_l)$} {\bf B4}}}};
(10,-3)* {\text{\makebox[1.9cm][r]{{\it big image} {\bf B5}}}};
"Htl"; "Htr" **\dir{-};
"Htr"; "Hbr" **\dir{-};
"Hbr"; "Hbl" **\dir{-};
"Hbl"; "Htl" **\dir{-};
"S1tl"; "S1tr" **\dir{-};
"S1tr"; "S1br" **\dir{-};
"S1br"; "S1bl" **\dir{-};
"S1bl"; "S1tl" **\dir{-};
"S2tl"; "S2tr" **\dir{-};
"S2tr"; "S2br" **\dir{-};
"S2br"; "S2bl" **\dir{-};
"S2bl"; "S2tl" **\dir{-};
"S4tl"; "S4tr" **\dir{-};
"S4tr"; "S4br" **\dir{-};
"S4br"; "S4bl" **\dir{-};
"S4bl"; "S4tl" **\dir{-};
"S5tl"; "S5tr" **\dir{-};
"S5tr"; "S5br" **\dir{-};
"S5br"; "S5bl" **\dir{-};
"S5bl"; "S5tl" **\dir{-};
"S6tl"; "S6tr" **\dir{-};
"S6tr"; "S6br" **\dir{-};
"S6br"; "S6bl" **\dir{-};
"S6bl"; "S6tl" **\dir{-};
\endxy
\caption{Logical structure of argument for Theorem \ref{main-theorem}\label{big-fig}}
\end{figure}

\emph{Step 1:} Given an $r$ as in the theorem, we can immediately apply Proposition \ref{seed-prop}, constructing a rational prime $l'$ and an $l'$-adic representation $r'$, satisfying the conclusions 1--5.

\emph{Step 2:}  We now apply Proposition \ref{geom-prop} taking $s=1$, $\fq_1=v_q$ and $r=2, l_1=l, l_2=l'$, and $K'=\Kbar^{\ker \bar{r}}$; and using $\bar{\rho}_1=\bar{r}$ and $\bar{\rho}_2=\bar{r}'$ (the semisimplification of the reduction of $r'$). 
Conditions 
\ref{geom-prop-cdx-l-cn}, \ref{geom-prop-cdx-l-mod}, \ref{geom-prop-cdx-unram-at-cl}, \ref{geom-prop-cdx-inertia} on $\bar{\rho}_1=\bar{r}$ and $l_1=l$ are satisfied by hypotheses 
C2, C3, C4, C5 respectively, together with conclusion 4 of step 1 which controls $\bar{r}$ ar $l'$. Next, I claim that the determinant of the polarization on $\Prim[l]_v$ matches the determinant of the polarization on $r$; this is from hypothesis B1, and the fact that $l$ splits in $F^*(n,N)$ which tells us $\Prim[l]_v$ has polarization with determinant a square by Proposition \ref{split-det-pol-prop}. Finally, we can use condition C6 to get the rest of condition \ref{geom-prop-cdx-dual}. 

Conditions \ref{geom-prop-cdx-l-cn} and \ref{geom-prop-cdx-l-mod} on $l_2=l'$ are satisfied by conclusion \ref{seed-prop-conclude-lprime} of Proposition \ref{seed-prop} applied in step 1; 
and conditions \ref{geom-prop-cdx-unram-at-cl} and \ref{geom-prop-cdx-inertia} on $\bar{\rho}_2=\bar{r}'$ are met respectively by conclusions 
\ref{seed-prop-conclude-rprime-unram}, \ref{seed-prop-conclude-inertia} of the same proposition.  Finally, condition \ref{geom-prop-cdx-dual} on $\bar{r'}$ is met since $r'$ is symplectic with multiplier $\eps^{1-n}$ (note that this automatically means that the determinant of the polarization will be -1, which is a square since $l'\equiv 1$ mod 4; this will match $\Prim[l']$ since $l'$ splits in $F^*(n,N)$).

We are left with a CM field $K_1$, a point $t\in T_0^{(M)}(K_1)$, and characters $\chi_l$ and $\chi_{l'}$
 satisfying the conclusions 1--5 of Proposition \ref{geom-prop}.

\emph{Step 3:} I claim that $(\Prim_{l',t}(\vec{h})\otimes\chi_l)|_{G_{K_1}}$, is automorphic of weight 0 and type $\{\Sp_n(1)\}_{\{\fQ|v_q\}}$. To check this, in the light of conclusion \ref{seed-prop-conclude-some-things-modular} of the Proposition in step 1, it suffices to check the conditions a--f given there. Conditions (a) and (f) are met by conclusions \ref{geom-prop-conclude-agree-mod}, \ref{geom-prop-conclude-steinberg} of the proposition in step 2, and conditions (d) and (e) are met by conclusion \ref{geom-prop-conclude-crys-ht}. Condition (b) is a simple geometric fact about our family established in Proposition \ref{geom-lemma} (point (\ref{geom-lemma-dual})). Finally, condition (c) is automatic since $\Prim_{l',t}$ is a piece of the cohomology of a variety and $\chi_l$ is finite order.

We can immediately deduce that  $\Prim_{l',t}(\vec{h})|_{G_{K_1}}$ itself is automorphic.

\emph{Step 4:} Since $\Prim_{l',t}$ and $\Prim_{l,t}$ are part of a compatible system, which are crystalline/unramified (as appropriate) at $l$ and $l'$ (because of conclusion \ref{geom-prop-conclude-good-reduction} of the proposition applied in step 2), the fact that $\Prim_{l',t}(\vec{h})|_{G_{K_1}}$ is automorphic implies $\Prim_{l,t}(\vec{h})|_{G_{K_1}}$ is also automorphic (of weight 0 and type $\{\Sp_n(1)\}_{\{\fQ|v_q\}}$).

\emph{Step 5:} I claim that $r|_{G_{K_1}}$, is modular of weight 0 and type $\{\Sp_n(1)\}_{\{w_q\}}$. We shall see this using Theorem \ref{cht-lifting-thm}. (Note that in applying this theorem we use the fact that $l>n$.) Conditions 1 and 2 are met by hypotheses B1, B2 respectively. Conditions 3 and 4 are both satisfied by condition B3, with $a=0$. For condition 5, hypothesis A (and the fact $w_q|v_q$) gives us what we need. Conditions 6 and 7 are met by hypotheses B4, B5 respectively. (For condition 7, we also use the fact that the field extension we made in step 2 was linearly disjoint from the fixed field of the kernel of $\bar{r}$.) Condition 8 comes from the fact that $(\Prim_{l,t'}(\vec{h})\otimes\chi'_l)|_{G_{K_1}} \equiv r$ mod $l$.

This completes the proof of Theorem \ref{main-theorem}.
\end{proof}

\section{A twisting argument}\label{sec:twisting}
\subsection{} In this section, we will briefly sketch an argument showing that condition (7) of our main theorem, Theorem \ref{main-theorem}, can be relaxed under an assumption that $\Q_l$ contains `enough roots of unity'. In particular, we shall sketch proof that:
\begin{corollary} \label{cor:twist}
Let $n,N$ be positive integers with $N\geq n+5$, $n$ even, and $N$ odd, and let $C(n,N)$ and $F^*(n,N)$ be as in Theorem \ref{orig-main-theorem}. Suppose then that $F$ is a CM field containing $\mu_N$ and $\mu_n$, and $l$ is a rational prime satisfying those conditions placed on it in Theorem \ref{orig-main-theorem}.

Suppose we are given a representation $r$ satisfying all the conditions placed on $r$ in Theorem \ref{orig-main-theorem} except condition (7) on the determinant of $\rbar$ need not hold. Suppose in addition that $\Q_l$ contains $\Q(\mu_{Mn})$, where $M$ is the order of the character $(\det r)\eps_l^{n(n-1)/2}$. (We know that his character has finite order since condition (4) tells us that $\det r$ is crystalline with Hodge-Tate number $n(n-1)/2$.) Then the conclusion of Theorem \ref{orig-main-theorem} still holds.
\end{corollary}
This means that if one has a compatible system of representations $r_l$ and wishes to apply Theorem \ref{orig-main-theorem} to \emph{some} representation in the family, one can usually do so without concern for condition (7). In particular, one notes that the characters $\det r_l$ form a compatible system, so the characters $(\det r_l)\eps_l^{n(n-1)/2}$ form a compatible system of finite-order characters; and in particular, they all have the same order $M$. Thus we can choose $l$ to be a rational prime which splits in $\Q(\mu_{Mn})$ and $F^*(n,N)$ and splits in any further fields which are convenient for the particular application one has in mind, and then apply the corollary.

We will need three facts from class field theory. The author is grateful to Brian Conrad for explaining a quick proof of the first of these facts.
\begin{lemma}\label{cft-lemma1}
Suppose $K$ is a CM field, and $\phi:G_K\rightarrow \overline{\mathbb{Q}}^{\times}$ is a finite order Galois character satisfying $\phi\phi^c=1$. Then we can write $\phi=\psi^c/\psi$ for some finite order character $\psi:G_K\rightarrow \overline{\mathbb{Q}}^{\times}$.
\end{lemma}
\begin{proof}
Let $K^+$ denote the maximal totally real subfield of $K$, and $D$ denote the group $\Q/\Z$ (upon which we will place a trivial $G_{K^+}$ action). We want $H^1(\Gal(K/K^+),H^1(G_K,D))=0$. We have a spectral sequence $E^{i,j}_2=H^i(\Gal(K/K^+),H^j(G_K,D))\Rightarrow H^{i+j}(G_{K^+},D)$. By Tate's theorem, $H^2(G_F,D)=0$ for any global field $F$, so $E^{0,2}_2=0$. Also, $E^{2,0}_2=H^2(\Gal(K/K^+),D)$, which vanishes by double periodicity of Tate cohomology for cyclic groups and the divisibility of $D$. We see the abutment in degree 2 vanishes, so $E^{1,1}_2$ must vanish provided $E^{0,3}_2$=0. But for any number field $K$, $H^3(G_K,.)=\prod_{v|\infty} (G_{K_v},.)$, so $H^3(G_K,D)=\prod_{v|\infty} (G_{K_v},D)$ which vanishes in our case since $K$ is totally complex.
\end{proof}

\begin{lemma}\label{cft-lemma2}
Suppose that $n$ is a positive integer, that $K$ is a number field containing $\mu_n$, and that $\phi:G_K\rightarrow\bar{\Q}$ is a finite order character of $G_K$. Then the obstruction to taking an $n$-th root of $\phi$ can be identified with an $n$-torsion element in the Brauer group of $K$.
\end{lemma}
\begin{proof} We have a short exact sequence of abelian groups $0\to(\frac{1}{n}\Z)/\Z\to \Q/\Z \overset{\times n}\to \Q/\Z\to 0$, and we may place a trivial $G_K$ action on them and then take the long exact sequence in cohomology, part of which reads
$ H^1(\Q/\Z ,G_K)\overset{\times n}\to H^1(\Q/\Z ,G_K) \to H^2(((1/n)\Z)/\Z, G_K) \to H^2(\Q)/\Z, G_K)$,
and so
$ \Hom(\Q/\Z ,G_K)\overset{\times n}\to \Hom(\Q/\Z ,G_K) \to H^2(((1/n)\Z)/\Z, G_K) \to 0$,
using Tate's result that $H^2(\Q)/\Z, G_K)=0$ for $K$ a number field. But the first two groups in the sequence both isomorphic to the group of finite order characters of $G_K$, with the map between them being the $n$-th power map. Thus the obstruction to finding an $n$th root is the cokernel of this map, which from the exact sequence is $H^2(((1/n)\Z)/\Z, G_K)$, which is $\cong H^2(\mu_n, G_K)$ (since $\mu_n\subset K$, so $((1/n)\Z)/\Z\isoto \mu_n$ as groups with a Galois action), Then we finally have that $H^2(\mu_n, G_K)\isoto \Br(K)[n]$, and we are done.\footnote{To see this, we take the long exact sequence in cohomology associated to $1\to \mu_n \to \overline{K}^\times \overset{x\mapsto x^n}\to \overline{K}^\times \to 1$, and use Hilbert's Theorem 90 and the fact that $\Br(K)=H^2(\overline{K}^\times,G_K)$}
\end{proof}

\begin{lemma}\label{cft-lemma3}
Suppose $n$ is a positive integer, $F$ is a CM field containing $\mu_n$, $l$ is a rational prime, $F^{\mathrm{avoid}}$ is an extension of $F$, and $\chi: G_F to \bar{\Q}_l^\times$ is a Galois character, which is finite order,  unramified at $l$, and satisfies $\chi \chi^c = 1$.  Then we can find a CM extension $F''$ of $F$, linearly disjoint from $F^{\mathrm{avoid}}$, and a finite order character $\psi: G_{F''} \to \bar{\Q}_l^\times$ with:
\begin{itemize}
\item $\psi$ unramified at $l$
\item $\psi \psi^c = 1$, and
\item $\psi^n = \chi|_{G_{F'}}$
\end{itemize}
Moreover, if we are given a set of primes at which $\chi$ is unramified, we can arrange that $\psi$ is again unramified at those primes.  Finally, if F is Galois over some smaller field $F_0$, we can arrange that $F''$ is too.
\end{lemma}
\begin{proof}[Proof sketch]

Next note $\chi$ can be written as $\phi/\phi^c$ for some character $\phi$, by Lemma \ref{cft-lemma1}. Our next goal is to find some CM extension $F'$ of $F$, linearly disjoint from $\Fbar^{\ker \bar{r}}$ over $F$, over which $\phi$ has an $n$th root. By Lemma \ref{cft-lemma2}, since $n$ is prime and so we certainly don't have $8|n$, the obstruction to $\phi$ having an $n$th root can be identified with an $n$-torsion element in $\alpha\in\Br(K)$; writing $\alpha_v$ for the image of $\alpha$ in $\Br(K_v)$ for each place $v$ of $K$, then since $\Br(K)\subset\bigoplus_v \Br(K_v)$, $\alpha_v$ is 0 for almost all $v$, and we see we can kill $\alpha$ by making any global field extension which induces at each place where $\alpha_v$ is nontrivial a local extension whose degree is divisible by the order of $\alpha_v$. We can make these local extensions in a way that keeps us linearly disjoint from any extension we like, and also keeps the field we work with CM, and finally is done in a way which keeps the extension Galois over $F_0$. (We make an extension to the totally real subfield which will give large enough local extensions everywhere we need them.) 

Then, over the extension for which $\phi$ has an $n$th root ($\kappa$, say), we take $\psi=\kappa/\kappa^c$.
\end{proof}

\begin{proof}[Sketch proof of Corollary \ref{cor:twist}]
Suppose that $n$, $N$, $C$, $F^*$, $F$, $l$ and $r$ are as in the statement of the corollary. Let $\chi$ be $(\det \rbar)\bar{\eps}_l^{-(1-n)n/2}$, and let $\tilde{\chi}$ be the Teichmuller lift of $\chi$. By hypothesis, $r$ satisfies conditions (1--6) and (8--11) of Theorem \ref{main-theorem}, and taking determinants of condition (2) we see that $\chi \chi^c = 1$. Using Lemma \ref{cft-lemma3} above we pass to an extension field $F''$ linearly disjoint from  $\Fbar^{\ker \bar{r}}(\zeta_l)$ over $F$ where we can find some finite order character $\psi: G_{F''} \to \bar{\Q}_l^\times$ with  $\psi$ unramified at $l$, $\psi \psi^c = 1$, $\psi^n = \chi|_{G_{F''}}$ and $\psi$ unramified at primes of $\cL$. Moreover since $\chi$ has order $M$, $\psi$ has order at most $nM$, and hence we see that $\chi$ can be taken to have values in $\Q_l^\times$, since $l$ splits in $\Q(\mu_{Mn})$. 

We claim that $(\psi^{-1} \otimes r)$ satisfies all the conditions (1-11) of Theorem \ref{main-theorem} (including condition (7)). Conditions (1), (4), (5), and (6) are trivial. Condition (2) is immediate given the fact that $\psi\psi^c=1$. Condition (3) is immediate since the Bella\"iche-Chenevier sign is unaffected by twisting. For condition (7), we see that $\det (\rbar \otimes \psi^{-1}) = (\det \rbar) (\psi^{-n })= (\det \rbar)\chi^{-1} = \bar{\eps}_l^{(1-n)n/2}$. For condition (8) we use the fact that if the image of a representation is `big', then the same is true for any twist (this is a consequence of \cite[Proposition 2.2]{snowden-wiles}) and the fact that $F''$ is linearly disjoint from $\Fbar^{\ker \bar{r}}(\zeta_l)$ over $F$. For condition (9) we use the fact that $F''$ is linearly disjoint from $\Fbar^{\ker \bar{r}}(\zeta_l)$ over $F$ and the fact that $\psi$ is unramified at $l$. For condition (10) we use again the fact that $\psi$ is unramified at $l$, and for condition (11) we use the fact that the polarization is unaffected by twisting.

Thus, applying the original theorem \ref{main-theorem}, we can find a further extension $F'$ of $F''$, still linearly disjoint from $\Fbar^{\ker \bar{r}}$, such that we get that $\psi^{-1} \otimes r$ is automorphic over $F'$. Then we're done, since a twist of an automorphic representation is automorphic. (In the case where we are given a field $F_0$ such that $F''$ must be Galois over $F_0$, we can arrange this by ensuring that $F''$ is Galois over $F_0$ by using the last sentence of Lemma \ref{cft-lemma3}. 
\end{proof}


\begin{thebibliography}{99}
\bibitem{blggt} T.~Barnet-Lamb, T.~Gee, D.~Geraghty, R.~Taylor, \emph{Potential automorphy and change of weight}, in preparation.
\bibitem{bc} J.~Bellaiche and G.~Chenevier, \emph{The sign of Galois representations attached to automorphic forms for unitary groups}, preprint available; to appear in \emph{Stabilisation de la formule des traces, varietes de Shimura et applications arithmetiques}, in preparation.
\bibitem{cht} L.~Clozel, M.~Harris and R.~Taylor, \emph{Automorphy for some l-adic lifts of automorphic mod l representations}, to appear in proc. IHES.
\bibitem{dmos} P.~Deligne, J.~S.~Milne, A.~Ogus, K.-Y.~Shih, \emph{Hodge cycles, motives and Shimura varieties}, LNM 900, Springer 1982.
\bibitem{gs} I.~M.~Gessel and R.~P.~Stanley, `Algebraic enumeration', in R.~L.Graham, M.~Gr\"otschel and L.~Lov\'asz, eds., \emph{Handbook of combinatorics}, Vol. 2. Elsevier, (1995)
\bibitem{ghk} R.~Guralnick, M.~Harris and N.~M.~Katz, \emph{Automorphic Realization of residual Galois representations}, preprint.
\bibitem{parisbook} M.~Harris, ed. \emph{Stabilization of the trace formula, Shimura varieties, and arithmetic applications.}, to appear.
\bibitem{hsbt} M.~Harris, N.~Shepherd-Barron and R.~Taylor, \emph{A family of Calabi-Yau varieties and potential automorphy}, to appear in Ann. Math.
\bibitem{k} N.~Katz, \emph{Another look at the Dwork family}, to appear in Manin Festschrift.
\bibitem{k-book} N.~Katz, \emph{Exponential sums and differential equations}, Annals of Math. Study 125, Princeton Univ. Press, 1990.
\bibitem{mvw} C.~R.~Matthews, L.~N.~Vaserstein, and B.~Weisfeiler, \emph{Congruence properties of Zariski dense subgroups I}, Proc. Lon. Math. Soc. 48 (1984), 514--532.
\bibitem{nori} M.~Nori, \emph{On subgroups of $GL_n(\F_p)$}, Invent. Math. 88 (1987), 257--275.
\bibitem{t-potmod} R.~Taylor, \emph{Remarks on a conjecture of Fontaine and Mazur}, Journal of the Institute of Mathematics of Jussieu 1 (2002), 1--19.
\bibitem{t-mod} R.~Taylor, \emph{Automorphy for some l-adic lifts of automorphic mod l representations, II}, to appear in proc. IHES.
\bibitem{serre-alr} J.-P.~Serre, \emph{Abelian $l$-adic representations and elliptic curves}, W. A.  Benjamin, Inc., New York-Amsterdam 1968 xvi+177 pp. 
\bibitem{shin} S.~W.~Shin, \emph{Galois representations arising from some compact Shimura varieties}, to appear in Ann. Math.
\bibitem{snowden-wiles} A.~Snowden, A.~Wiles, \emph{Bigness in compatible systems}, preprint
\bibitem{we-js} A.~Weil, \emph{Jacobi sums as ``Gr\"ossencharaktere''}, Trans. Amer. Math. Soc. 
73, (1952). 487-495. 

\end{thebibliography}
\end{document}